\documentclass[11pt]{amsart}

\usepackage{amsmath,graphics}
\usepackage{amsfonts,amssymb}
\usepackage{amscd}
\usepackage{mathrsfs}
\usepackage[all]{xy}
\usepackage{accents}
\usepackage{array}
\usepackage{tikz}

\usetikzlibrary{arrows,decorations.markings}

\theoremstyle{plain}

\newtheorem*{theorem*}{Theorem}
\newtheorem*{lemma*} {Lemma}
\newtheorem*{corollary*} {Corollary}
\newtheorem*{proposition*} {Proposition}
\newtheorem{theorem}{Theorem}[section]
\newtheorem{lemma}[theorem]{Lemma}

\newtheorem{proposition}[theorem]{Proposition}

\newtheorem{definition}{Definition}
\newtheorem*{notation}{Notation}

\theoremstyle{definition}

\newtheorem{convention}[theorem]{Convention}

\def \Z {\mathbb{Z}}
\def \C {\mathbb{C}}
\def \P {\mathbb{P}}

\def \bn{\begin{enumerate}}
\def \en{\end{enumerate}}
\def \bdm{\begin{displaymath}}
\def \edm{\end{displaymath}}

\def \bp{\begin{proof}}
\def\ep{\end{proof}}
\def\be{\begin{equation}}
\def\ee{\end{equation}}

\def\mfp{\mathfrak{p}}

\def\mfg{\mathfrak{g}}
\def\mfn{\mathfrak{n}}
\def\mfm{\mathfrak{m}}
\def\mfh{\mathfrak{h}}

\def\mcD{\mathcal{D}}
\def\sgla{\mfg=\bigoplus_{k\in \Z}\mfg_k}

\def\msc{\mathscr{C}}
\def\msk{\mathscr{K}}
\def\mca{\mathcal{A}}

\newcommand\DistTo{\xrightarrow{
   \,\smash{\raisebox{-0.1ex}{\ensuremath{\scriptstyle\sim}}}\,}}

\begin{document}

\title{EQUIVARIANT COMPACTIFICATIONS OF A NILPOTENT GROUP BY $G/P$}
\author{Daewoong Cheong}
\address{Korea Institute for Advanced Study \\
207-43 Cheongryangri 2-dong\\
Seoul, 130-722, Korea} \email{daewoongc@kias.re.kr}

\date{March 12, 2015}

\begin{abstract}
Let $G$ be a simple complex algebraic group, $P$ a parabolic subgroup of $G$ and $N$ the unipotent radical of $P.$  The so-called equivariant compactification of $N$ by $G/P$ is given by an action of $N$ on $G/P$ with a dense open orbit isomorphic to $N$. In this article, we investigate
how many such equivariant compactifications there exist. Our result says that
there is a unique equivariant compactification of $N$ by $G/P$, up to isomorphism, except $\P^n$.
\end{abstract}

\maketitle

\section{Introduction}

Throughout, we will work over complex numbers.  Let $H$ be an algebraic group. Then a projective variety $M$  is called an {\it equivariant compactification} of $H$ if there is an algebraic $H$-action $H\times M\rightarrow M $ with a Zariski open orbit $O$, equivariantly biregular to $H$. In this situation, the action of $H$ is called an {\it Equivariant Compactification structure} (for short, {\it  EC-structure}) on $M$.
An isomorphism between two EC-structures $A_i:H\times M\rightarrow M$ for $i=1,2$ consists of  an automorphism $F:H\rightarrow H$ and a biregurlar morphism $\Psi:M\rightarrow M$ such that the diagram
$$\xymatrix{H\times M \ar[r]^-{A_1}\ar[d]_{(F,\Psi)}&M\ar[d]^{\Psi}\\
H\times M\ar[r]^-{A_2}&M}$$ commutes.
\smallskip

In  \cite{HT}, Hassett and Tschinkel studied EC-structures on $\P^n$
 of the vector group $\mathbb{G}_a^n$. They showed that there is a correspondence between EC structures of the vector group and Artin local algebras of length $n$. Then, by classifying such algebras, they showed that there are  finitely (resp. infinitely) many isomorphism classes of EC-structures on $\P^n$ for $n\geq2$ (resp. $n\geq 6)$.

Arzhantsev considered EC-structures of a vector group on a homogeneous space $G/P$  for a semisimple algebraic group $G$ and a parabolic subgroup $P$ of $G$ (\cite{Ar1}). His result, restricted to simple groups $G$,  says that $G/P$ is an equivariant compactification of a vector group exactly when  the unipotent radical $P_u$ of $P$ is commutative, or when $P_u$ is not commutative and there is a  pair $(\widetilde{G},\widetilde{P})$ such that  $G/P=\widetilde{G}/\widetilde{P}$  and the unipotent radical $\widetilde{P}_u$ of $\widetilde{P}$ is commutative; see Subsection \ref{subsection-auto} for $\widetilde{G}$ and $\widetilde{P}$. Arzhantsev also  raised the question of how many such EC-structures there are on $G/P$, which was  mentioned again  by Arzhantsev and   Sharoyko (\cite{AS1}).

Very recently, Fu and Hwang  proved  a more general result on EC-structures of a vector group and posed the question whether the analogue of their result holds for other linear algebraic groups (\cite{FH2}). Roughly, their result  says that if $M$ is a Fano manifold, different from $\P^n$, of Picard number $1$ having smooth VMRT, then all EC-structures on $M$ are isomorphic. See Section \ref{section:VMRT}  for the definition of VMRT.  There are various manifolds that satisfy the necessary condition on VMRT. They include all irreducible Hermitian symmetric homogeneous spaces and some non-homogeneous spaces (see Proposition of $6.14$ of \cite{FH1} for examples of non-homogeneous spaces). In particular, their result answers the question of  Arzhantsev and   Sharoyko (\cite{Ar1}, \cite{AS1})  on the classification of EC-structures of a vector group on $G/P$  for $P$ maximal.

Lastly, Devyatov also studied EC-structures on $G/P$ of a vector group, using relevant Lie algebra representations. To classify EC-srtuctures, he  first set up a correspondence between EC-structures on $G/P$ and multiplications with a certain property on Lie algebra representations, and then solved the initial EC-problem by classifying these multiplications.

In this article, we generalize the results on vector groups above to some non-commutative nilpotent groups $N$. To make a precise statement, let us fix some notations. Let $\mfg=\bigoplus_{k\in \Z}\mfg_k$ be a simple graded Lie algebra. Then there is a triple $(G,P,N)$ of groups associated with  $\mfg=\bigoplus_{k\in \Z}\mfg_k$, where $G$ is an (adjoint) algebraic group with Lie algebra $\mfg,$ $P$ is a parabolic subgroup of $G$ with Lie algebra $\mfp=\bigoplus_{k\geq 0}\mfg_k$, and $N$ is a unipotent radical of $P$, so that the Lie algebra of $N$ equals $\mfn=\bigoplus_{k>0}\mfg_k.$ Then our main theorem can be stated as follows.

\begin{theorem}
Let $\mfg=\bigoplus_{k}\mfg_k$ be a simple graded Lie algebra, and let $G,P$ and $N$ be given as above. Suppose that  the homogeneous space $G/P$ is  different from  $\P^n.$ Then there exists a unique EC-structure of $N$ on $G/P$, up to isomorphism.
\end{theorem}
We remark that existence of an EC-structure on $G/P$ follows directly from the Bruhat decomposition of $G/P$(\cite{BL1}).
 To prove uniqueness, we mostly follow a proving scheme Fu and Hwang made in \cite{FH2}. A main difference lies in which tool to use in order to obtain an extension of a locally defined map to the entire space, which is the most essential  in proving the uniqueness; Fu and Hwang used the Cartan-Fubini extension theorem which is applicable only to Fano manifolds of Picard number $1$, while we use the Yamaguchi's result on the prolongation of simple graded Lie algebras. For reader's convenience, let us give a bit more detailed sketch of proof of the uniqueness.
 \smallskip

 First of all, we divide all simple graded Lie algebras $\sgla$ except for two cases  $(A_l,\{\alpha_1\})$ and  $(C_l,\{\alpha_1\})$  into three types $I,II$ and $III$ (Subsection \ref{subsection:prolongation}).
 \begin{enumerate}
 \item The simple graded Lie algebras of type $I$ are all but the simple graded Lie algebras of types $II$ and $III$ right below.
 \item The simple graded Lie algebras of type $II$ are ones of depth one, or contact gradation, but not isomorphic with $(A_l,\{\alpha_1,\alpha_k\})$.
 \item The simple graded Lie algebras of type $III$ are ones isomorphic with  $(A_l,\{\alpha_1,\alpha_k\})$ or  $(C_l,\{\alpha_1,\alpha_k\})$.
 \end{enumerate}

 We note that  for a simple graded Lie algebra $\sgla,$ there are the natural differential system $D\subset T(G/P)$ and the so called VMRT $\msc$ on $G/P$, and,
in particular, for the type $III$ there is the natural decomposition $D=D^{(1)}\oplus D^{(2)}$ on $G/P.$ See Subsections \ref{subsection:differential systems}, \ref{subsection:VMRT} and \ref{subsection:decomposition} for the differential system $D$, VMRT $\msc$, and the decomposition $D=D^{(1)}\oplus D^{(2)}$, respectively.
Let $\mathcal{A}$ (resp. $\mca_\msc$) be the sheaf of Lie algebras of infinitesimal automorphisms preserving $D$ (resp. $\msc$) on $G/P$, and for simple graded Lie algebra of type $III,$ let $\mca_{\oplus}$ be the sheaf of Lie algebras of infinitesimal automorphisms on $G/P$ preserving the decomposition $D=D^{(1)}\oplus D^{(2)}$.
\smallskip

In order to prove the uniqueness, we suppose that there are two EC-structures of $N$ on $G/P$
$$A_i:N\times G/P\rightarrow G/P\hspace{0.08in}\textrm{for}\hspace{0.08in} i=1,2.$$
We only have to construct an isomorphism between two EC-structures $A_i$. As the first step toward a construction of the isomorphism,  we construct a local biholomorphic map on $G/P$ as follows.
  For each $i=1,2,$ $O_i$  be the Zariski open orbit for the EC-structure $A_i$, and, after fixing a point $x_i\in O_i$, define a biholomorphic map $a_i:N \rightarrow O_i$ by $a_i(h)=A_i(h,x_i)$. Then using two pull-backed differential systems $a_i^*(D|_{O_i})$ on $N,$ we construct a Lie algebra isomorphism $f:\mfn \rightarrow \mfn$ which induces a group isomorphism $F:N\rightarrow N$. Then, through the biholomorphic maps $a_i$, the group isomorphism $F$, in turn, gives rise to a biholomorphic map $\Psi:O_1\rightarrow O_2$ that  preserves $D$, $\msc,$ or the decomposition $D=D^{(1)}\oplus D^{(2)}$ for types $I, II$  or $III,$ respectively. This is a content of Proposition \ref{prop3}.

 As the second step, we need to extend the local biholomorphic $\Psi:O_1\rightarrow O_2.$ To do this, we first characterize $\mfg$ as a subalgebra of the Lie algebra of holomorphic vector fields on any connected open subsets $U\subset G/P$ as follows.
 \begin{proposition}(Propositions \ref{pro:typeI}, \ref{pro:typeII}, \ref{prop:type III})
Let $\sgla$ be a simple graded Lie algebra, $G/P$ its associated homogeneous space. Then for any connected open subset $U\subset G/P,$ the Lie algebra of sections $\mca(U)$ (resp. $\mca_\msc(U)$, $\mca_\oplus (U)$) is  isomorphic to $\mfg$ if $\sgla$ is type $I$ (resp. type II, type III).
\end{proposition}

  We note that since $\Psi:O_1\rightarrow O_2$ preserves $D$, $\msc,$ or the decomposition $D=D^{(1)}\oplus D^{(2)}$ for types $I, II$  or $III,$ respectively, the differential $d\Psi$ sends  $\mathcal{A}(O_1)$,
 $\mca_\msc(O_1)$ or $\mca_\oplus(O_1)$ isomorphically onto $\mathcal{A}(O_2)$,
 $\mca_\msc(O_2)$ or $\mca_\oplus(O_2)$ for types $I, II$  or $III,$ respectively, i.e., $d\Psi$ send $\mfg$ to $\mfg$ for all types $I, II$ and $III.$
 Furthermore, it turns out that  $d\Psi$ sends  a parabolic subalgebra $\mfp$ onto another parabolic subalgebra $\mfp^\prime$. Therefore we have an isomorphism $d\Psi:(\mfg,\mfp)\rightarrow (\mfg,\mfp^\prime)$, which induces an automorphism $\widetilde{\Psi}:G/P\rightarrow G/P^\prime$, extending $\Psi:O_1\rightarrow O_2.$ This is a content of Proposition \ref{coro1}.

As a final step, we show that the biholomorphic map $\widetilde{\Psi}$ is equivariant with respect to the actions $A_i$, which implies that $\widetilde{\Psi}$ is an isomorphism between two EC-structures $A_i$ on $G/P.$

{\it \textbf{Acknowledgements}.}  The author was
supported by National Researcher Program 2010-0020413 of NRF.  The author would like to thank A. Ottazzi for answering some questions and B. Fu for useful discussions and suggestions. He especially thanks to J-M Hwang for suggesting the problem, providing a decisive idea for solving the problem, and helping him understand the idea by much explanation.

\section{Simple graded Lie algebras and its associated homogeneous spaces}

\subsection{Gradations of simple Lie algebras}
In this subsection, we collect basics on complex simple Lie algebras. We refer to \cite{Ya} for more detailed accounts.

\smallskip
 A $gradation$ of a Lie algebra $\mfg$ is a direct sum $\mfg=\bigoplus_{k\in \Z}\mfg_k$ of subspaces of $\mfg$ such that $[\mfg_r,\mfg_s]\subset \mfg_{r+s}$.
 A graded Lie algebra is a Lie algebra equipped with a gradation $\sgla$.
 Throughout the paper, we will restrict ourselves to (complex) simple graded Lie algebra $\sgla$ satisfying an additional condition
 $$\mfg_k=[\mfg_{k+1},\mfg_{-1}]\hspace{0.08in}\textrm {for \hspace{0.05in}all}\hspace{0.08in} k< -1.$$  Such simple graded Lie algebras can be characterized as follows.

Fix a Cartan subalgebra $\mfh \subset \mfg$ and a simple root system $\triangle:=\{\alpha_1,....\alpha_l\}$ of the root system $\Phi$ with respect to $\mfh$, and denote  by $\Phi^+$ the set of positive roots.  Let $\triangle_1$ be a nonempty subset of $\triangle$. For each $k\geq 0$, let $\Phi_k^+$ be the set of positive roots $\alpha\in \Phi^+$ that can be written as $\alpha=\sum_{i=1}^l n_i(\alpha)\alpha_i$ for $n_i(\alpha)$ with $\sum_{\alpha_i\in \triangle_1}n_i(\alpha)=k.$ Then we put $$\mfg_0= \mfh\oplus \bigoplus_{\alpha \in \Phi_0^+}(\mfg_\alpha\oplus \mfg_{-\alpha}),$$
$$\mfg_k=\bigoplus_{\alpha \in \Phi_k^+} \mfg_\alpha, \hspace{0.1in}\mfg_{-k}=\bigoplus_{\alpha \in \Phi_k^+} \mfg_{-\alpha} \hspace{0.15in}(k>0).$$ Then it is easy to check that  the direct sum $\sgla$ forms a gradation of $\mfg$ and satisfies the condition
 $\mfg_k=[\mfg_{k+1},\mfg_{-1}]$ for all $k<-1.$
 When $\mfg$ is of Lie type $X_l$, we often write  $(X_l,\triangle_1)$ for the simple graded Lie algebra $\mfg=\bigoplus_{k\in \Z}\mfg_k$ obtained in this way.
The converse is true by Theorem $3.12$ of \cite{Ya}. Namely, if $\mfg=\bigoplus_{k\in \Z}\mfg_k$  is a gradation of  a simple Lie algebra $\mfg$ satisfying $\mfg_k=[\mfg_{k+1},\mfg_{-1}]$ for each $k<-1$ and $\mfg $ is of Lie type $X_l,$ then it is conjugate to $(X_l,\triangle_1)$ for a subset $\triangle_1$ of the simple root system $\triangle$ of $X_l$ .
\smallskip

A gradation $\mfg=\bigoplus_{k\in \Z}\mfg_k$
is symmetric in the sense that $\mfg_{k}\ne 0$ if and only if $\mfg_{-k}\ne 0.$
More precisely, the dimension of $\mfg_{k}$ is equal to the dimension of $\mfg_{-k}$.
This follows from the fact that the restriction of Killing form $B$ to $\mfg_k\times \mfg_{-k}$ is nondegenerate (Lemma 3.1 of \cite{Ya}). Therefore for a simple graded Lie algebra $\mfg$, there is a unique integer $\mu >0$ such that $\mfg_k\ne 0$ for any $-\mu\leq k\leq \mu,$ and
$\mfg=\bigoplus_{k=-\mu}^{\mu}\mfg_{k}.$ Such an integer $\mu$ is called the $depth$ of $\mfg=\bigoplus_{k\in \Z}\mfg_k$.

\smallskip
Let $\mfg=\bigoplus_{k\in \Z}\mfg_k$ be a simple graded Lie algebra, and let $\mfg^i=\bigoplus_{k\geq i}\mfg_k$.
Then subalgebra $\mfg^0$ is  a parabolic subalgebra $\mfp$ associated with $\triangle_1,$ and  the subalgebra $\mfg^{1}$ is the nilradical $\mfn$ of $\mfp.$ Let $G$ be the adjoint algebraic group of $\mfg$, and let $P$ and $N$ be  subgroups of $G$ whose Lie algebras are  $\mfp$ and $\mfn,$ respectively.
Then $P$ is a parabolic subgroup of $G$ and the nilpotent subgroup $N$ is a unipotent radical of $P$. In this situation,  we say that $G/P$ is  {\it associated with} the graded simple Lie algebra $\mfg=\bigoplus_{k\in \Z}\mfg_k$.

\subsection{Examples of simple graded Lie algebras}
We give an explicit description of gradations of some simple graded Lie algebras which will be used in later sections.
\subsubsection{$A_l$ type ($l\geq 1$)}
For $\mfg=\mathfrak{sl}(l+1)$, we choose Cartan subalgebra $\mathfrak{h}$ consisting of all diagonal matrices $D(a_1,...,a_{l+1})$ in $\mathfrak{sl}(l+1)$, where $D(a_1,...,a_{l+1})$ is the diagonal matrix with the entries $a_1,...,a_{l+1}$ on the diagonal. For $i=1,...,l+1,$ let $\lambda_i$ be linear forms defined by $\lambda_{i}:D(a_1,...,a_{l+1})\mapsto a_i.$ Let $E_{i,j}$ be the $(l+1)\times (l+1)$ matrix whose $(i,j)$-th entry is $1$ and  other entries are all $0.$
Then we have $$[H,E_{i,j}]=(\lambda_i-\lambda_j)(H)(E_{i,j})\hspace{0.12in \textrm{for}\hspace{0.05in} H\in \mathfrak{h}}.$$
Thus, we have the root system  $\Phi=\{\lambda_i-\lambda_j\in \mathfrak{h}^* \hspace{0.03in}|\hspace{0.03in}1\leq i, j\leq l+1, i\ne j\}.$  Take a simple root system $\triangle =\{\alpha_1,...,\alpha_l\}$, where $\alpha_i=\lambda_i-\lambda_{i+1}$. For  $(A_{1},\{\alpha_k\}),$ the gradation of  $(A_{1},\{\alpha_k\})$ is given by $\mathfrak{sl}(l+1)=\mfg_{-1}\oplus\mfg_0\oplus\mfg_1$; \begin{displaymath} \mfg_{-1}=\left\{ \left(\begin{array}{cc}
0 & 0\\
C& 0\\
\end{array}\right)\hspace{0.04in}|\hspace{0.04in} C\in M(j,k)\right \},
\hspace{0.13in} \mfg_{1}=\left\{ \left(\begin{array}{cc}
0 & D\\
0 & 0\\
\end{array}\right)\hspace{0.04in} |\hspace{0.04in} C\in M(k,j)\right\},
\end{displaymath}

\begin{displaymath}
 \mfg_{0}=\left\{ \left(\begin{array}{cc}
A & 0\\
0 & B\\
\end{array}\right)\hspace{0.04in} |\hspace{0.04in} A\in M(k,k),\hspace{0.07in} B\in M(j,j) \hspace{0.07in}\textrm{and}\hspace{0.07in} \textrm{tr} A +\textrm{tr} B=0\right\},
\end{displaymath}
where $j=l+1-k$ and $M(p,q)$ denotes the set of $p\times q$ matrices.
This gradation can be described schematically by the diagram;
\begin{center}
\tiny{\begin{tikzpicture}[scale=0.9]

   \draw (0,0) rectangle (2,2);
   \draw [very thin](0,1.4)-- (2,1.4);
   \draw [very thin](0.6,0)-- (0.6,2);

   \draw (0.3,1.7) node{$0$};
   \draw (0.3,0.7) node{$-1$};
   \draw (1.3,1.7) node{$1$};
   \draw (1.3,0.7) node{$0$};

    \draw (0.3,2.2) node{$k$};
    \draw (1.3,2.2) node{$j$};
    \draw (-0.2,1.7) node{$k$};
     \draw (-0.2,0.7) node{$j$};

     \draw (3.7, 1.1) node{$(j=l+1-k)$};
    \end{tikzpicture}}
\end{center}

The diagram of $(A_{1},\{\alpha_{k_1}, \alpha_{k_2}\})$ is obtained by superposing two diagrams corresponding to $(A_{1},\{\alpha_{k_1}\})$ and $(A_{1},\{\alpha_{k_2}\});$
\begin{center}
\tiny{\begin{tikzpicture}[scale=0.9]

   \draw (0,0) rectangle (2.4,2.4);
   \draw [very thin](0,1.8)-- (2.4,1.8);
   \draw [very thin](0.6,0)-- (0.6,2.4);

   \draw (0.3,2.1) node{$0$};
   \draw (0.3,0.9) node{$-1$};
   \draw (1.5,2.1) node{$1$};
   \draw (1.5,0.9) node{$0$};

   \draw (3.4,0) rectangle (5.8,2.4);
   \draw [very thin](3.4,0.6)-- (5.8,0.6);
   \draw [very thin](5.2,0)-- (5.2,2.4);

   \draw (0.3,2.1) node{$0$};
   \draw (0.3,0.9) node{$-1$};
   \draw (1.5,2.1) node{$1$};
   \draw (1.5,0.9) node{$0$};

   \draw (4.3,1.5) node{$0$};
   \draw (4.3,0.3) node{$-1$};
   \draw (5.5,1.5) node{$1$};
   \draw (5.5,0.3) node{$0$};

  \draw (7.8,0) rectangle (10.2,2.4);
   \draw [very thin](8.4,0)-- (8.4,2.4);
   \draw [very thin](7.8,1.8)-- (10.2,1.8);
\draw [very thin](7.8,0.6)-- (10.2,0.6);
   \draw [very thin](9.6,0)-- (9.6,2.4);

 \draw (8.1,2.1) node{$0$};
   \draw (8.1,1.2) node{$-1$};
   \draw (8.1,0.3) node{$-2$};
   \draw (9,2.1) node{$1$};
    \draw (9,1.2) node{$0$};
   \draw (9.1,0.3) node{$-1$};
   \draw (9.9,2.1) node{$2$};
   \draw (9.9,1.2) node{$1$};
   \draw (9.9,0.3) node{$0$};

   \draw (2.9, 1.2) node{+};
    \draw (7, 1.2) node{=};

    \draw (1.2,-0.5) node{$(A_l,\{\alpha_{k_1}\})$};
   \draw (4.6,-0.5) node{$(A_l,\{\alpha_{k_2}\})$};
   \draw (9,-0.5) node {$(A_l,\{\alpha_{k_1}, \alpha_{k_2}\})$};
    \end{tikzpicture}}
\end{center}
In general, the diagram of $(A_{1},\{\alpha_{k_1},..., \alpha_{k_m}\})$ is obtained
by superposing $m$ diagrams of $(A_{1},\{\alpha_{k_1}\}),..., (A_{1},\{\alpha_{k_m}\})$.
\subsubsection{$C_l$ type ($l\geq2$)}
Let $(V,\langle,\rangle)$ be a symplectic vector space of dimension $2l$. Let $\mfg=\mathfrak{sp}(V).$ Choose a symplectic basis $\{e_1, ...,e_l,f_1,...,f_l\}$ of $V$ such that $\langle e_i,e_j\rangle=\langle f_i,f_j\rangle=0$ and $\langle  f_i, e_{l+1-j}\rangle=\delta_{i,j}$ for $i,j=1,...,l.$ Then using this basis, elements of $\mfg$ can be written as $2l\times 2l$-matrices of the form
\begin{displaymath}X=X(A,B,C)=\left(\begin{array}{cc}
A & B\\
C& -A^\prime\\
\end{array}\right).
\end{displaymath}
 Here  $A,B,C$ are $l\times l$ matrices, and $B,C$ are matrices such that $B=B^\prime$ and $C=C^\prime$, where  $()^\prime$ means taking ``transpose" with respect to the anti-diagonal line.

 In this case, the root systems is $\Phi=\{\lambda_i-\lambda_j \hspace{0.05in}(i\ne j), \hspace{0.08in}\pm(\lambda_i+\lambda_j)\hspace{0.05in} (1\leq i\leq j\leq l)\}$. Note that the roots $\lambda_{i}-\lambda_j$, $\lambda_i+\lambda_j$ and $-\lambda_i-\lambda_j$ have the root vectors $X(E_{i,j},0,0)$, $X(0,F_{l+1-j},0)$ and $X(0,0,F_{l+1-i,j})$, respectively, where $F_{i,j}:=E_{i,j}+E_{i,j}^\prime.$ By putting
 $\alpha_i=\lambda_i-\lambda_j$ for $1\leq i\leq l-1$ and $\alpha_l=2\lambda_{l},$ we take a simple root system $\triangle=\{\alpha_i \hspace{0.04in}|\hspace{0.04in} i=1,...,l\}.$

Then we see that the gradation of $(C_1,\{\alpha_k\})$ is given by the following  diagrams;
\begin{center}
\tiny{\begin{tikzpicture}[scale=0.9]

   \draw (0,0) rectangle (2.4,2.4);
   \draw [very thin](0,1.8)-- (2.4,1.8);
   \draw [very thin](0.6,0)-- (0.6,2.4);
   \draw[very thin] (0,0.6)-- (2.4,0.6);
    \draw[very thin] (1.8,0)-- (1.8,2.4);

   \draw (0.3,2.1) node{$0$};
   \draw (0.3,1.2) node{$-1$};
   \draw (0.3,0.3) node{$-2$};
   \draw (1.2,2.1) node{$1$};
   \draw (1.2,1.2) node{$0$};
   \draw (1.2,0.3) node{$-1$};
   \draw (2.1,2.1) node{$2$};
   \draw (2.1,1.2) node{$1$};
   \draw (2.1,0.3) node{$0$};

    \draw (-0.2,2.1) node{$k$};
    \draw (0.3,2.6) node{$k$};
    \draw (2.1,2.6) node{$k$};
    \draw (-0.2,0.3) node{$k$};

   \draw (3.9,1.2) node{$(1\leq k<l)$};

    \draw (6.3,0) rectangle (8.7,2.4);
   \draw [very thin](7.5,0)-- (7.5, 2.4);
   \draw [very thin](6.3,1.2)-- (8.7,1.2);

   \draw (6.9,1.8) node{$0$};
   \draw (6.9,0.6) node{$-1$};
   \draw (8.1,1.8) node{$1$};
   \draw (8.1,0.6) node{$0$};
   \draw (9.6,1.2) node{$(k=l)$};

    \end{tikzpicture}}
\end{center}
As in Lie type $A_l,$ the diagram of $(C_{1},\{\alpha_{k_1},..., \alpha_{k_m}\})$ is obtained
by superposing $m$ diagrams of $(C_{1},\{\alpha_{k_1}\}),..., (C_{1},\{\alpha_{k_m}\})$.
\subsection{Automorphism group of $G/P$}\label{subsection-auto}
Let $G/P$ be the homogeneous space associated with  $\mfg=\bigoplus_{k\in \Z}\mfg_k$.
Then it is well-known that $G$ (of adjoint type) coincides with the identity component $\mathrm{Aut}^0(G/P)$ of the group of automorphisms of $G/P$ except when  $\mfg=\bigoplus_{k\in \Z}\mfg_k$ is isomorphic with $(C_l, \{\alpha_1\})$ ($l\geq 2$), $(B_l, \{\alpha_l\})$ ($l\geq 3$) or $(G_2, \{\alpha_1\})$; see \cite{On}.
In these exceptions, the Lie algebra of the algebraic group  $\mathrm{Aut}^0(G/P)$ is of type $A_{2l-1}$, $D_{l+1}$ or $B_3$, respectively. Let $\widetilde{G}=\mathrm{Aut}^0(G/P).$ Then there is a parabolic subgroup $\widetilde{P}\subset \widetilde{G}$ such that $G/P=\widetilde{G}/\widetilde{P}$. Indeed, the simple graded Lie algebra with which $\widetilde{G}/\widetilde{P}$ is associated is $(A_{2l-1},\{\alpha_1\})$, $(D_{l+1},\{\alpha_{l+1}\})$ or $(B_3,\{\alpha_1\})$, respectively.

\section{Differential systems}\label{section:Differential system}
In this section, we give an overview on  differential systems on complex manifolds $M$.
\subsection{Generalities on differential systems}\label{subsection:differential systems}
Materials in this subsection are taken from \cite{Ya}. For more detailed accounts or proofs, also see \cite{T2}.
 A  {\it differential system} on $M$ is a  subbundle $D$ of the tangent bundle $TM$ on a complex manifold $M$. We will simply write $(M,D)$ for a differential system $D$ on $M$. A differential system $(M,D)$ is completely integrable if $[\mcD,\mcD]\subset \mcD$, where $\mcD$ is the sheaf of sections of $D.$ In this paper, we only deal with non-integrable differential system $(M,D)$ which we simply refer to as  a differential system.
A differential system  $(M,D)$ gives rise to the $k$-th weak derived system $\mcD^k$ which  is a subsheaf of sections of $TM$  defined inductively as follows.  Define $$\mcD^{-1}=\mcD,\hspace{0.08in} \mcD^{k}=\mcD^{k+1}+[\mcD,\mcD^{k+1}]\hspace{0.1in}\textrm{for}\hspace{0.08in}k<-1.$$ In general, $\mcD^k$ is not locally free. But if $\mcD^k$ is locally free for each $k\leq -1$, then the differential system $D$ is called {\it regular}, and we denote by $D^k$ the bundle corresponding to $\mcD^k$. We often call $D^k$ the {\it $k$-th weak derived system of $D$} for $k\leq -1$, too.

\begin{proposition}\label{pro:regular-Differential}(Proposition $1.1$ of \cite{T2}) Let $(M,D)$ be a regular differential system. Then  the weak derived system of $D$ satisfies
 \begin{enumerate}
 \item There exists a unique integer $\mu>0$ such that

$$ TM= D^{-\mu}\supset
D^{-\mu+1}\supset \cdots \supset D^{-1}=D, \hspace{0.08in}\textrm{and}\hspace{0.08in} TM \ne D^{-\mu+1}. $$
\item
$[\mcD^k, \mcD^l]\subset \mcD^{k+l}$ for all $k,l<0.$
\end{enumerate}
\end{proposition}
To a regular differential system $(M,D),$  we associate a nilpotent graded Lie algebra $\mfm(x)$ for each $x\in M$  defined as follows. Put $\mfg_{-1}(x)=D^{-1}(x)$ and $\mfg_k(x)=D^k(x)/D^{k+1}(x)$ for $k<-1.$ Then define
\begin{displaymath}\mfm(x)=\bigoplus_{k=-1}^{-\mu}\mfg_k(x).\end{displaymath}
 The Lie bracket defined on local vector fields around $x$ induces a bracket product on the graded vector space $\mfm(x)$ which gives a graded Lie algebra structure on $\mfm(x)$ and satisfies $\mfg_k(x)=[\mfg_{k+1}(x),\mfg_{-1}(x))]$ for $k<-1$. Note that $\mfm(x)$ depends on the chosen $D$, and  $\mfm(x)$ is nilpotent by  Proposition \ref{pro:regular-Differential}.
The nilpotent graded Lie algebra $\mfm(x)$ is called the {\it symbol algebra of $D$ at $x$}. Let $\mfm$ be a fundamental graded Lie algebra of $\mu$-th kind, i.e.,
$\mfm=\bigoplus_{k=-1}^{-\mu}$ be a nilpotent graded Lie algebra such that  $\mfg_k=[\mfg_{k+1},\mfg_{-1}]$ for $k<-1.$
 Then the differential system $D$ on $M$ is called {\it of type} $\mfm$ if the symbol algebra $\mfm(x)$ is isomorphic with $\mfm$ for each $x\in M.$

\subsection{Natural regular differential system on homogeneous spaces} Let $\sgla$ be a simple graded Lie algebra, and $G/P$ its associated homogeneous space.
Then the tangent bundle $T(G/P)$ is identified with the associated bundle $G\times_P(\mfg/\mfg^0).$ Under this identification,  $D:=G\times_P(\mfg^{-1}/\mfg^0)$ forms a regular differential system on $G/P,$ and the weak derived system $D^k$ for $k\leq -1$ induced by $D$ is identified with $G\times_P(\mfg^k/\mfg^0)$. Another way of constructing $D$ is as follows. First, identify $\mfg$ with the Lie algebra of left invariant vector fields on $G$. Then the subspace $\mfg^{-1} \subset \mfg$ defines a left invariant subbundle of the tangent bundle $TG$. Note that this subbundle is preserved under the right action of $P$ on $G.$ Therefore $\mfg^{-1}$ induces a $G$-invariant differential system $D$ on $G/P.$
The following is a standard fact about a differential system on $G/P$ associated with $\mfg=\bigoplus_{k\in \Z}\mfg_k$.
\begin{proposition}(\cite{CN})\label{Natural}
 Let $\mfg=\bigoplus_{k=-\mu}^{\mu}\mfg_k$ be a simple graded algebra such that $\mfm=\bigoplus_{k=-1}^{-\mu}\mfg_k$ is fundamental.
Then the associated homogeneous space $G/P$  carries a natural regular differential system $D\subset T(G/P)$ of type $\mfm$.
\end{proposition}

\subsection{Local picture of the natural differential system $D$ on $G/P$}
  Let  $\mfg=\bigoplus_{k\in \Z}\mfg_k$ be a simple graded Lie algebra . We describe a local picture of the differential system $D$ on the associated homogeneous space $G/P$. We fix $S=G/P$ in case that there may be  a different presentation  $G/P^\prime$ for a different parabolic subgroup $P^\prime$ involved in some explanation.
 First of all, let us construct `principal'  dense open subsets of $S$ indexed by parabolic subgroups of $G$. Note that there is a principal dense open subset of $S$ associated with the given $\sgla$ (or its parabolic $P$). To be precise,  let $N^-$  a
 subgroup of $G$ with Lie algebra $\mfm=\bigoplus_{k<0}\mfg_k,$  and  $\iota_P:N^-\hookrightarrow G/P$ the restriction of the natural projection $\pi_P:G\rightarrow G/P$. Let  $M(\mfm)\subset S= G/P$ be the image of $\iota_P$. Then $M(\mfm)$ is a dense open subset of  $G/P$.

  To construct general principal open subsets of $S$, let $P^\prime$ be a parabolic subgroup of $G,$  so that $P^\prime=hPh^{-1}$ for some $h\in G.$
 The relation $P^\prime=hPh^{-1}$ induces an isomorphism $\tau_{P^\prime}:G/P^\prime\rightarrow S$ defined by $\tau_{P^\prime}(gP^\prime)=(h^{-1}gh)P$ for $gP^\prime \in G/P^\prime$, and  a new gradation $\mfg=\bigoplus_{k\in \Z}\mathfrak{g}^\prime_k,$
where  $\mathfrak{g}^\prime_k=h\mfg_kh^{-1}.$ Then $M(\mfm^\prime)$ constructed as above for the new gradation  $\mfg=\bigoplus_{k\in\Z}\mathfrak{g}^\prime_k$ is an open subset of $G/P^\prime$. By abuse of notation, let $M(\mfm^\prime)\subset S$ denote the image of $\tau_{P^\prime}:G/P^\prime\rightarrow S $.

 Now, given  $\sgla$, we  construct a model  differential system $D_\mfm$ on the manifold $N^{-}$ defined as follows (p. $420$ of \cite{Ya}).
 Identifying $\mfm$ with the Lie algebra of left invariant vector fields on $N^-$,
     $\mfg_{-1}$ defines a left invariant subbundle $D_\mfm$ of $T(N^-)$.  Then  $D_{\mfm}$ is a regular differential system on $N^-$ of type $\mfm.$
  Furthermore, we see that the differential system $D_{\mfm}$ on $N^-$ is identified with the differential system $ D|_{M(\mfm)}$ on $M(\mfm)\subset S$ via the isomorphism $\iota_P:N^{-}\rightarrow M(\mfm),$ where  $D$ is the natural differential system on $G/P$ of Proposition \ref{Natural} associated with $\mfg=\bigoplus_{k\in \Z}\mathfrak{g}_k.$

 More generally, let $P^\prime $ and $\mfg=\bigoplus_{k\in \Z}\mfg_k^\prime$ be given as above. Let  $N^{\prime -}:=(N^{\prime})^{-}$ be the subgroup of $G$ with  Lie algebra $\mfm^\prime=\bigoplus_{k<0}\mfg_k^\prime$.  As above,  for $\mfg=\bigoplus_{k\in \Z}\mfg_k^\prime$  (or $P^\prime$),  we can construct a model differential system on $D_{\mfm^\prime}$ on $N^{\prime -}$, which is identified with the differential system $D^{\prime}|_{M(\mfm^\prime)}$ on $M(\mfm^\prime)\subset G/P^\prime,$ where  $D^\prime$ is the natural differential system on $G/P^\prime$ associated with $\mfg=\bigoplus_{k\in \Z}\mfg_k^\prime$.
 But since the differential system  $D^\prime$ on $G/P^\prime$ is identified with  the differential system  $D$ on $S$ via $\tau_{P^\prime}:G/P^\prime\rightarrow S $, the differential system $D_{\mfm^\prime}$ on $N^{\prime -}$ is identified with the differential system  $D|_{M(\mfm^\prime)}$ on $M(\mfm^\prime) \subset S$ via the composition of maps $\tau_{P^\prime}\circ \iota_{P^\prime}$, where $\iota_{P^\prime}:N^{\prime -}\hookrightarrow G/P^\prime$ is the restriction of the natural projection $\pi_{P^\prime}:G\rightarrow G/P^\prime$ to $N^{\prime -}$.
 This construction shows how the differential systems on principal open subsets of $S$ look like.
 \section{Varieties of minimal rational tangents}\label{section:VMRT}
This section is devoted to giving an overview of a variety of minimal rational tangent of $M$. We refer to \cite{Hw2} for details on VMRT.
Throughout  this section, we will assume that $M$ is a Fano manifold of dimension $n$ and of Picard number $1$.
\subsection{Minimal rational curves} A {\it parametrized rational curve}, or simply a {\it rational curve} on $M$ is a morphism $f:\P^1\rightarrow M$ that is birational over its image. It is well-known that for each $x$ on a Fano manifold $M$ of Picard number $1$, there is  a rational curve   $f:\P^1\rightarrow M$ through $x$ (\cite{Mo1} and \cite{Ko1}).  Recall that a vector bundle on $\P^1$ is isomorphic to a direct sum of line bundles. A rational curve  $f:\P^1\rightarrow M$ is {\it free} if the pull-back of the tangent bundle  $TM$ of $M$ splits as
$$f^*TM\stackrel{\sim}{=}\mathcal{O}(a_1)\oplus \cdots \oplus \mathcal{O}(a_n)$$ for all nonnegative integers $a_i$. A free rational curve  $f:\P^1\rightarrow M$ is {\it minimal} if its anti-canonical degree, i.e., the degree of $f^*(K_X^{-1})$ is minimal. A {\it minimal rational curve} on $M$ is a rational curve that can be obtained as a deformation of minimal free rational curves. A {\it minimal rational component} is  a component of the Chow space of $M$ whose members are minimal rational curves. Note  that there are only finitely many minimal rational components (Section $1.3$ of \cite{Hw2}).
A free rational curve  $f:\P^1\rightarrow M$ is {\it standard} if \be \label{dimension p} f^*TM\stackrel{\sim}{=}\mathcal{O}(2)\oplus [\mathcal{O}(1)]^p\oplus \mathcal{O}^{n-1-p},\ee
where $p+2$ is the anti-canonical degree of $f.$ By the so-called bend-and-break of Mori, a generic member of a minimal rational component is standard (\cite{Ko1}).

\subsection{Variety of minimal rational tangents}\label{subsection:VMRT}
Choose a minimal rational component $\mathscr{K}$, and for a general element $x\in M,$ let $\msk_x$  be the normalization of the Chow space of members of $\msk$ through $x.$ Then it was proved in \cite{Mo1} that $\msk_x$ is the union of finitely many  smooth algebraic varieties of dimension $p$, where $p$ is in (\ref{dimension p}). Note that a generic member of each component of $\msk_x$ is a standard rational curve that is smooth at $x.$ Thus, by associating to a member of $\msk_x$ smooth at $x$  its tangent, we define a rational map called the {\it tangent map} at $x$ $$\Theta_x:\msk_x\rightarrow \P(T_xM).$$  Let $\msc_x$ be the strict image of $\Theta_x$. Then $\msc_x$ is called the {\it variety of minimal rational tangents}, VMRT for short, at $x.$  We will denote by $\msc$ the variety obtained by taking the union of $\msc_x.$ The following proposition makes it possible to compute the VMRT at a general point for various manifolds.

\begin{proposition}(Proposition of \cite{Hw2})\label{Pro:VMRT}
Suppose that $M$ can be embedded in a projective space $\P^N$ so that for each $x\in M,$ $M$ contains a line of $\P^N$ through $x.$ Then the tangent map at a general point $y\in M$ is an embedding, and $\msc_y$ is smooth.
\end{proposition}

\subsection{Examples}
We give examples of VMRT for some homogeneous spaces.
\subsubsection{Projective space} For  $\P^n,$ the Chow space $\msk_x$ of all lines through $x$ is isomorphic to $\P^{n-1}$ for every $x\in \P^n.$ From Proposition \ref{Pro:VMRT}, we have the isomorphism $$\Theta_x:\msk_x=\P^{n-1} \DistTo \P(T_x\P^n).$$ So we have $\msc_x=\P(T_x\P^n).$
\subsubsection{Grassmannian} Let $\mathbb{G}(m,N)$ be the Grassmannian of $m$-dimensional subspaces $W$ in a vector space $V$ of dimension $N$. Let $x$ be a point of $\mathbb{G}(m,N)$ corresponding to a subspace $W.$ Under the Pl\"{u}cker embedding, $\mathbb{G}(m,N)$ satisfies the hypothesis of Proposition \ref{Pro:VMRT}, and so $\Theta_x$ is an embedding. To find $\msk_x,$ we note that
a line $L$ on $\mathbb{G}(m,N)$ is completely determined by an $(m-1)$-dimensional subspace $W^\prime \subset W$ and an $(m+1)$-dimensional subspace $W^{''}\subset V$ such that $W\subset W^{''},$ and points on the line $L$ correspond to subspaces of $V$ containing $W^\prime$ and contained in $W^{''}$. But $(m-1)$-dimensional subspaces of $W$ and $(m+1)$-dimensional subspaces of $V$ containing $W$ are parametrized by $\P(W^*)$ and $\P(V/W)$, respectively. Thus, $\msk_x$ is isomorphic to $\P^{m-1}\times \P^{N-m-1}$. Then we can identify the VMRT $\msc_x$ at $x$ through the tangent map
$$\Theta_x:\msk_x=\P^{m-1}\times\P^{N-m-1}\hookrightarrow \P(T_xM)=\P(W^*\otimes V/W).$$ Note that  $\mathbb{G}(m,N)$ is the homgeneous space associated with the depth $1$ simple graded Lie algebra  $(A_{N-1},\{\alpha_m\})$.
\smallskip

For other examples, let us take simple graded Lie algebras $(B_l,\{\alpha_2\})$ ($l\geq 3$) and $(D_l,\{\alpha_2\})$  ($l\geq 4$). Then they have the contact gradation, and  the associated homogeneous spaces  $G/P$ have $\P^1\times Q^{2l-5}$ and $\P^1\times Q^{2l-6}$ as VMRTs, respectively, where $Q$ denotes the quadric.
Details on this can be found in  \cite{Hw1}.

\section{Prolongation of $\mfm$}
In this section, we study the prolongation introduced originally by Tanaka
which provide us with a main tool to prove our result. References are \cite{Ya} and \cite{T2}.

\subsection{Definition of the prolongation} \label{subsection:prolongation}

Let $\mfm=\bigoplus_{k<0} \mfg_k$ be a fundamental graded Lie algebra of $\mu$-th kind. We construct  a graded Lie algebra depending on $\mfm$ $$\mfg(\mfm)=\bigoplus_{k\in \Z}\mfg_k(\mfm)$$ as follows.
For $k<0,$ put $\mfg_{k}(\mfm)=\mfg_k.$
Let $\mfg_0(\mfm)$ be the Lie algebra of gradation preserving derivations on the graded Lie algebra $\mfm=\bigoplus_{k<0} \mfg_k$. For $u \in \mfg_0(\mfm)$ and $X\in \mfm$, define $[u,X]=-[X,u]=u(X).$
We easily see that this bracket satisfies the Jacobi identity on the direct sum $\bigoplus_{k\leq 0}\mfg_k(\mfm)$, and so $\bigoplus_{k\leq 0}\mfg_k(\mfm)$ is a graded Lie algebra.
 For $k>0,$ we proceed inductively. Suppose $\mfg_p(\mfm)$ are defined for all $p\leq k.$
Then let $\mfg_{k+1}(\mfm)$ be the vector space of degree $(k+1)$ linear maps $u:\mfm \rightarrow \bigoplus_{j\leq k}\mfg_j(\mfm)$ satisfying $$u([X,Y])=[u(X),Y]+[X,u(Y)] \hspace{0.1in}\mathrm{for}\hspace{0.05in} X,Y \in \mfm.$$
So we have a direct sum $\mfg(\mfm)=\bigoplus_{k\in \Z}\mfg_k(\mfm)$ of vector spaces.
Now we define a Lie bracket  on $\mfg(\mfm)=\bigoplus_{k\in \Z}\mfg_k(\mfm)$. For  $X\in \mfm$ and $u \in \mfg_r(\mfm)$ with $r\geq 0$, define $[u, X]=u(X).$
 For $u_1\in \mfg_r(\mfm)$ and $u_2 \in \mfg_s(\mfm),$ define a bracket  $[u_1,u_2]$ to be a degree $(r+s)$  linear map $[u_1,u_2]:\mfm\rightarrow \bigoplus_{ j\leq r+s-1}\mfg_j(\mfm)$ given by $$[u_1,u_2](X)=[u_1(X), u_2]+[u_1,u_2(X)].$$  We can easily check that this bracket satisfies the Jacobi identity, and so  $\mfg(\mfm)$ becomes a graded Lie algebra.  The graded Lie algebra $\mfg(\mfm)$ is called   the {\it  prolongation of $\mfm.$}  There is a more general notion of  prolongation: Let $\mfg_0$ be a subalgebra of the Lie algebra $\mfg_0(\mfm).$ For $k\geq 1,$ define a subspace of $\mfg_k(\mfm)$ inductively as
 $$\mfg_k=\{u\in \mfg_k(\mfm)\hspace{0.04
 in} | \hspace{0.07in}[u, \mfg_{-1}]\subset \mfg_{k-1}\}.$$
 Then $\mfg(\mfm,\mfg_0):=\mfm\oplus\bigoplus_{k\geq 0}\mfg_k$ is a graded subalgebra of $\mfg(\mfm)$.  $\mfg(\mfm,\mfg_0)$ is called the {\it prolongation of $(\mfm,\mfg_0).$}

The prolongation $\mfg(\mfm)$ satisfies the following property (\cite{T2}):
\be\label{property 3}
\mathrm{For}\hspace{0.05in}  k\geq 0, \hspace{0.05in}\mathrm{if}\hspace{0.05in}\hspace{0.05in}  X\in \mfg_k(\mfm)\hspace{0.05in} \mathrm{and}\hspace{0.05in} [X,\mfm]=0 \hspace{0.05in}, \hspace{0.08in}\mathrm{then} \hspace{0.05in}X=0.\ee

 The following result of Yamaguchi is essential in proving our result.
\begin{proposition}(Theorem $5.2$ of \cite{Ya})\label{prop1}
Let   $\mfg=\bigoplus_{k\in \Z}\mfg_k$   be a simple graded Lie algebra satisfying $\mfg_k=[\mfg_{k+1},\mfg_{-1}]$ for $k<-1$. Then $\mfg=\bigoplus_{k\in \Z}\mfg_k$ is the prolongation of $\mfm=\bigoplus_{k<0}\mfg_k$ except for the following three cases.
\begin{enumerate}
\item $\mfg=\bigoplus_{k\in \Z}\mfg_k$ is of  depth one, i.e., $\mfg=\mfg_{-1}\oplus \mfg_0\oplus \mfg_1.$
\item  $\mfg=\bigoplus_{k\in \Z}\mfg_k$ has a contact gradation, i.e., of depth two with $\mfg_{-2}$ having dimension one.
\item  $\mfg=\bigoplus_{k\in \Z}\mfg_k$ is isomorphic with $(A_l, \{\alpha_1,\alpha_k\})$ $(1< k<  l)$ or $(C_l, \{\alpha_1, \alpha_l\})$.
\end{enumerate}
Furthermore, $\mfg=\bigoplus_{k\in \Z}\mfg_k$ is the prolongation of $(\mfm,\mfg_0)$ except when  $\mfg=\bigoplus_{k\in \Z}\mfg_k$ is isomorphic with  $(A_l,\{\alpha_1\})$ or $(C_l,\{\alpha_1\})$.
\end{proposition}
\smallskip

\begin{definition} We divide simple graded Lie algebras $\mfg=\bigoplus_{k\in \Z}\mfg_k$ not isomorphic with $(A_l,\{\alpha_1\})$ or $(C_l,\{\alpha_1\})$ into three types $I, II$ and $III$ as follows.
\begin{enumerate}
\item $\mfg=\bigoplus_{k\in \Z}\mfg_k$ is  {\it of type $I$} if it is the prolongation of $\mfm=\bigoplus_{k<0}\mfg_k$.
\item $\mfg=\bigoplus_{k\in \Z}\mfg_k$  is
{\it of type $II$}  if it is of depth one or contact gradation, and not isomorphic with  $(A_{l},\{\alpha_1,\alpha_l\})$.
\item $\mfg=\bigoplus_{k\in \Z}\mfg_k$ is {\it of type $III$} if it is isomorphic with  $(A_{l},\{\alpha_1,\alpha_k\}) $  $(2\leq k\leq l$) or  $(C_{l},\{\alpha_1,\alpha_l\}) $.
\end{enumerate}

 The  associated homogeneous space  $G/P$ is said to be of types $I$, $II$ or $III$ if  $\mfg=\bigoplus_{k\in \Z}\mfg_k$ is of types $I$, $II$ or $III$, respectively.
\end{definition}

\subsection{Infinitesimal automorphisms and the prolongation}\label{subsection:auto-prplongation}
Let  $\mfg=\bigoplus_{k=-\mu}^{\mu}\mfg_k$ be a simple graded Lie algebra.  We will explain how local vector fields on $G/P$ preserving the differential system $D$  are related to the prolongation.
 Recall that the differential system $(G/P, D)$ is locally isomorphic with the differential system $(N^{\prime -}, D_{\mfm^\prime}).$  Thus we may work with $(N^{\prime -}, D_{\mfm^\prime})$ to give a local picture of $D$ on $G/P.$
 We  only give a sketch of a relationship between these two objects.  We refer to Section $2$ of \cite{Ya} for a complete description for the relation.  We begin with the `special chart' $(N^-,D_\mfm).$
\smallskip

Let $\omega: N^-\rightarrow \mfm$  be the Maurer-Cartan form on $N^-$.
Given a vector field $X$ on $N^-,$ we have a $\mfm$-valued function $f_X:
N^-\rightarrow \mfm$ defined by $f_X(h)=\xi(X_h)$ for $h\in N^-$. By the gradation  $\mfm=\bigoplus_{k=-\mu}^{-1}\mfg_k,$ we can write $$f_X=\sum_{k=-\mu}^{-1}f_X^k,$$ where $f_X^k \in \mfg_k.$ Furthermore, if a vector field $X$ on $N^-$ is an infinitesimal automorphism of $(N^-, D_\mfm)$, then we can
associate to the vector field $X$ a sequence of functions $$\{f_X^k:N^-\rightarrow\mfg_k(\mfm)\}_{k\geq 0},$$ where $f_X^k$ are inductively defined; we omit the explicit definition of $f_X^k$; see Page $427$ of \cite{Ya}.
Therefore, given an infinitesimal automorphism $X$ of $(N^-,D_\mfm)$, we obtain $\prod_{k\geq-\mu} \mfg_k(\mfm)$-valued function $f_X$ defined by
\be \label{eq4:prolongation}
f_X:
=\prod_{k\geq -\mu} f_X^k.\ee
We call $f_X
=\prod_{k\geq -\mu} f_X^k$ the {\it presentation of $X$ with respect to $\mfm$}.
  The function $f_X$ encodes higher derivative of $X$ in the sense that the sequence $\{f_X^k\}_{k\geq -\mu}$ satisfies, for each $k,$
$$df^k_X=\sum_{r=-\mu}^{-1}[f_X^{k-r},\xi^r].$$
\smallskip

 Conversely, given $a=\sum _{k=-\mu}^la^k$ of $\mfg(\mfm)$ with $a^k\in \mfg_k(\mfm)$ and $h_0\in N^-$, there is a unique infinitesimal automorphism $X$ preserving $D_\mfm$  such that $f_X^k(h_0)=a^k$ for $k\leq l,$ and $f_X^k\equiv0$ for $k>l$.
Indeed, $f_X$ (and so X) exists as a solution to the following differential equation for $\mfg_k(\mfm)$-valued functions $u^k=f_X^k$ $(-\mu\leq k \leq l)$
\be \label{differential eq} du^k=\sum_{k<s\leq l}[u^s, \xi^{k-s}] \hspace{0.1in}\mathrm{for}\hspace{0.04in}k=-\mu,...,l,\ee
subject to the initial condition $u^k(h_0)=a^k\in \mfg_k(\mfm).$

\smallskip
 We remark that the solution $X$ to the differential equation (\ref{differential eq}) is unique: For the infinitesimal automorphism $X$ constructed above, and any $g_0\in N^-$ with $g_0\ne h_0$, let $b=X(g_0)\in \mfg(\mfm)$. Then an infinitesimal automorphism $Y$ on $N^-$ constructed as above for $b\in \mfg(\mfm)$ and $g_0\in N^-$ coincides with the infinitesimal automorphism $X$ on $N^-.$

Recall that $\mathcal{A}$ is the sheaf of Lie algebras of infinitesimal automorphisms on $N^-$ preserving $D_\mfm.$
 \begin{definition}For  a fixed $h_0\in N^-,$ define $\psi_{h_0}: \mfg(\mfm)\rightarrow \mca(N^-)$ by $\psi_{h_0}(a)=X$, where $X$ is an infinitesimal automorphism constructed as above. When $\psi_{h_0}(a)=X$, we say that $X$ is {\it determined by $a$ at $h_0.$}
\end{definition}

\begin{convention}\label{convention1}
So far we have constructed an infinitesimal automorphism $X$ on $N^-$ determined by $a\in \mfg(\mfm)$ at $h_0\in N^-$, and  the map $\psi_{h_0}:\mfg(\mfm) \hookrightarrow\mathcal{A}(N^-)$.
When identifying $(N^{-},D_\mfm)$ and $(M(\mfm), D|_{M(\mfm)})$,
we use the same notations for corresponding objects: an infinitesimal automorphism $X$ on $M(\mfm)$ determined by $a\in \mfg(\mfm)$ and $x_0\in M(\mfm)$, and the map $\psi_{x_0}:\mfg(\mfm) \hookrightarrow\mathcal{A}(M(\mfm))$, where $x_0=\iota_P(h_0)$ for $\iota_{P}:N^-\DistTo M(\mfm)$.
\end{convention}
Note that we can do the same construction of the above objects $X$ and $\psi_{x_0}$ on other principal open subsets $M(\mfm^\prime)$ corresponding to $P^\prime$ by working with $(N^{\prime -}, D_{\mfm^\prime})$.

\subsection{Extension of infinitesimal automorphisms}

\begin{lemma}\label{extension}
 Let $\mfg=\bigoplus_{k\in \Z}\mfg_k$ be a simple graded Lie algebra, and let $S=G/P$ be its associated homogeneous space.   Let $M(\mfm)$ and $M(\mfm^\prime)$ be the principal dense open subset of $S$ corresponding to parabolic subgroups $P$ and $P^\prime$, respectively. If $X$ is an infinitesimal automorphism preserving $D$ on $M(\mfm)$ determined by $a=\sum a^k\in \mfg$ at $x_0\in M(\mfm),$ then $X$ extends to $M(\mfm^\prime)$.
   \end{lemma}

\begin{proof}
 Choose $h\in G$ such that $P^\prime=hPh^{-1}$, so that $\mfg=\bigoplus_{k\in\Z}\mfg_k^\prime$, $\mfm^\prime$ and $N^{\prime -}=(N^\prime)^-$ are given as before. Let us work with two `charts' $(N^-, D_\mfm)$ and $(N^{\prime -}, D_{\mfm^\prime})$.
  Let  $f_X^\mfm$ and $f_X^{\mfm^\prime}$ be the  presentations of $X$ with respect to $\mfm$ and $\mfm^\prime$, respectively.
 Fix $x_1\in M(\mfm)\cap M(\mfm^\prime)$. Let $g_1$ and  $g_1^\prime$ be the points of $N^-$ and  $N^{\prime -}$, respectively,  corresponding to $x_1$ (so we have $ g_1^\prime=hgh^{-1})$.
 Put $b:=f^\mfm_X(g_1)$.  Then we see that  from the construction of $X$, $f_X^\mfm$ is  $\mfg$-valued function on $N^-$, and hence $b$ belongs to $\mfg.$
 Thus, the vector field $X$ is determined by $b\in \mfg$ at $g_1\in N^-$ for the chart $(N^-, D_\mfm)$, too. Note that, by the `coordinate change', the element $b$ of $\mfg$ on the chart $N^-$ changes into the element $b^\prime:=hbh^{-1}$ of $\mfg$ on the chart $N^{\prime -}$. Therefore, letting $X^\prime$ be an infinitesimal automorphism on $N^{\prime -}$ determined by $b^\prime \in \mfg$ at $g_1^\prime\in N^{\prime -}$, by the uniqueness of solutions to the differential equation (\ref{differential eq}), we have $$X^\prime_{M(\mfm)\cap M(\mfm^\prime)}=X_{M(\mfm)\cap M(\mfm^\prime)}$$ Therefore the vector field $X$ extends to a vector field, denoted $X$, on $M(\mfm)\cup M(\mfm^\prime)$.
\end{proof}

\subsection{Prolongation of infinitesimal automorphisms for $G/P$ of type I}

 \begin{proposition}(Corollary $5.4$ of \cite{Ya})\label{pro:typeI}
 Let $S=G/P$ be a homogeneous space of type $I$, and let $U$ be any connected open subset of $S.$ Then $\mathcal{A}(U)=\mfg.$
\end{proposition}

\begin{proof}
First, suppose that $U\subset G/P$ is an  open subset of $M(\mfm)$ (with $o\in U$).
Let  $X\in \mathcal{A}(U)$. Then $f_X$ is a $\mfg(\mfm)$-valued function on $U$. Indeed, by the hypothesis that $\mfg(\mfm)=\mfg$ is finite dimensional,
 there is some $l$ such that if $k\geq l$, then $\mfg_k(\mfm)=0$. Thus if $k\geq l,$ then $f_X^k(x)=0$ for all $x\in U.$ Thus $f_X$ is $\mfg(\mfm)$-valued function. Now fix $x$, say $x=o.$ Let $a=X(o).$ It is obvious that $X=\psi_o(a).$ Since $X$ is arbitrary, $\psi_o$ is onto and hence an isomorphism.
Therefore we have shown that for $U\subset M(\mfm)$ $$\mathcal{A}(U)=\mathcal{A}(M(\mfm))=\mfg(\mfm)=\mfg.$$
Similarly, for any $U\subset M(\mfm^\prime)$ (corresponding to $P^\prime$), we have $$\mathcal{A}(U)=\mathcal{A}(M(\mfm^\prime))=\mfg.$$
Now if we are given a vector field $X\in \mathcal{A}(M(\mfm))$, then   using Lemma \ref{extension}, we can extend it to the entire space $G/P$. Thus the Lie algebra of global sections $\mathcal{A}(G/P)$ is isomorphic to $\mfg.$
Let $V$ be any connected open subset of $G/P$. Choose $U\subset M(\mfm^\prime)$ so that $U\subset V$. Then  we have $\mathcal{A}(V)=\mathcal{A}(G/P)=\mathcal{A}(U)=\mfg$. This completes the proof.
\end{proof}

\subsection{Prolongation of infinitesimal automorphism  for $G/P$ for  type II}
 Note that tangent space $T_o(G/P)$ is canonically  identified  with $\mfg_{-\mu}\oplus\cdots \oplus \mfg_{-1}$. Via this identification, the VMRT $\mathscr{C}_o$ at $o$ is  inside
$ \P(\mfg^{-1})$

\begin{lemma}\label{Lemma-contact}
 Given a  simple graded Lie algebra $\mfg=\bigoplus_{k\in \Z}\mfg_k$ of  $II,$  let  $G_0^\prime\subset \mathrm{Aut}(\mfg_{-1})$ be a group of automorphisms preserving $\mathscr{C}_o \subset \P\mfg_{-1}$, and $\mfg_0^\prime$ be the Lie algebra of $G_0^\prime$.  Then the Lie algebra $\mfg_0 $ coincides with the Lie algebra $\mfg_0^\prime$.

\end{lemma}

\begin{proof}
For simple graded Lie algebras with depth one, we refer to Section $3$ of \cite{FH1}. For the contact cases, see Section $2$ of \cite{Hw1}.
\end{proof}

\begin{proposition}\label{pro:typeII}
Let $G/P$ be a homogeneous space of type $II$, and $U$ be any connected open subset of $G/P$. Then $\mathcal{A}_{\msc}(U)=\mfg.$
\end{proposition}

\begin{proof}
 Suppose that $U\subset G/P$ is an open subset of $M(\mfm)$ (with $o\in U$).
 Then by the definitions of $\mfg_0^\prime$ and $\mca_\msc(U)$, $\psi_{x}:\mfg(\mfm)\hookrightarrow \mathcal{A}(U)$ descends to  $\psi_{x}^\msc:\mfg(\mfm,\mfg_0^\prime)\hookrightarrow \mathcal{A}_\msc(U)$ for any $x\in U$. But since $\mfg_0=\mfg_0^\prime$, we have that   $\mfg(\mfm,\mfg_0^\prime)=\mfg(\mfm,\mfg_0)=\mfg$ is finite dimensional. Then,
 as in the proof of Proposition \ref{pro:typeI}, by fixing $x$, say $x=o,$ $\psi_o^\msc$ is an isomorphism. Thus we have $$\mca_\msc(U)=\mca_\msc(M(\mfm))=\mfg(\mfm,\mfg_0^\prime)=\mfg(\mfm,\mfg_0)=\mfg.$$
 The proof for a general connected open subset $U\subset G/P$ is similar to the proof of Proposition \ref{pro:typeI}.
\end{proof}

\subsection{Prolongation of infinitesimal automorphisms for $G/P$ type III }\label{subsection:decomposition}

 Let  $\mfg=\bigoplus_{k\in \Z}\mfg_k$ be a simple graded Lie algebra of type $III,$ that is, $\mfg=\bigoplus_{k\in \Z}\mfg_k$ is isomorphic with  $(A_{l},\{\alpha_1,\alpha_k\}) $  $(2\leq k\leq l$) or  $(C_{l},\{\alpha_1,\alpha_l\}) $. Let $P_{1}$  and $P_{2}$ be the parabolic subgroups corresponding to the first root $\alpha_1$ and the second root $\alpha_k$, respectively. Then there are two natural projections from the associated homogeneous space $$\pi_i: G/P\rightarrow G/P_{i},$$  For $i=1,2,$ let $\mfg_{-1}^{(i)}$ be the kernel of the differential $d\pi_{i}:T_o(G/P)\rightarrow T_{o}(G/P_i)$. Then under the identification, $T_o(G/P)=\mfg_{-\mu}\oplus\cdots \oplus \mfg_{-1}$, $\mfg_{-1}$ decomposes into $\mfg_{-1}=\mfg_{-1}^{(1)}\oplus\mfg_{-1}^{(2)}.$
   To be concrete, $\mfg_{-1}^{(1)}$ and $\mfg_{-1}^{(2)}$ have the decompositions of root spaces; $$\mfg_{-1}^{(1)}=\bigoplus_{\alpha}\mfg_\alpha,\hspace{0.1in} \mfg_{-1}^{(2)}=\bigoplus_{\beta}\mfg_\beta,$$
 where for $(A_l,\{\alpha_1,\alpha_k\})$, $\alpha=\lambda_i-\lambda_j$ runs over $k+1\leq i\leq l+1$ and $2\leq j \leq k$, and $\beta=\lambda_i-\lambda_1$  over $2\leq i\leq k,$ and for  $(C_l,\{\alpha_1,\alpha_l\})$, $\alpha=-\lambda_i-\lambda_j$  over $2\leq i\leq j\leq l$ and $\beta=\lambda_i-\lambda_1$ over $2\leq i\leq l.$

More generally, by the construction of $D \subset TM$, $D$ decomposes into $D=D^{(1)} \oplus D^{(2)}$ such that $(D^{(i)})_o=\mfg_{-1}^{(i)}$ for $i=1,2.$

 \begin{lemma}\label{Lemma1}
   For  a simple graded Lie algebra $\mfg=\bigoplus_{k\in \Z}\mfg_k$  of type $III$, let $\mfg^\prime_0 \subset \mfg_0(\mfm)$ be the Lie algebra of elements
preserving the decomposition $\mfg_{-1}=\mfg_{-1}^{(1)}\oplus \mfg_{-1}^{(2)}.$ Then we have $\mfg_0=\mfg_0^\prime.$
\end{lemma}

\begin{proof}
 We refer to Section $5$ of \cite{Do1} for a proof for  $\mfg=\bigoplus_{k\in \Z} \mfg_k$ isomorphic to $(A_{l},\{\alpha_1,\alpha_k\}) $ ($2\leq k\leq l-1$) or $(C_l,\{\alpha_1, \alpha_l\})$.

 For the case  $(A_l,\{\alpha_1,\alpha_l\}),$ we give an elementary proof. Let $\mathrm{Der}_0(\mfm)$ be the algebra of gradation preserving derivations on $\mfm$, and let $\mfg_0^\prime \subset \mathrm{Der}_0(\mfm)$ be the algebra of derivations preserving the decomposition $\mfg_{-1}=\mfg_{-1}^{(1)}\oplus \mfg_{-1}^{(2)}.$
We will show that  $\mfg_0^\prime=\mfg_0$.
 Since it is obvious that   $\mfg_0\subset \mfg_0^\prime$, it is enough to show that $\mathrm{dim} \hspace{0.02in}\mfg_0^\prime=\mathrm{dim}\hspace{0.02in}\mfg_0$. For this, first note that since $\mfg_{-2}$ is generated by $\mfg_{-1},$
 $\varphi\in \mathrm{Der}_0(\mfm)$ is completely determined by its values on $\mfg_{-1}.$ Therefore the assignment $\varphi \mapsto \varphi|_{\mfg_{-1}}$
 defines an embedding of vector spaces
 $$\iota: \mfg_0^\prime\hookrightarrow  \mathrm{End} (\mfg_{-1}^{(1)})\oplus \mathrm{End} (\mfg_{-1}^{(2)}).$$ Now we will identify the equations defining the subspace  $\mfg_0^\prime$ in $\mathrm{End} (\mfg_{-1}^{(1)})\oplus \mathrm{End} (\mfg_{-1}^{(2)}).$ Note that these defining equations come from the derivation conditions among generators of $\mfg_{-1}$. So we fix a basis of $\mfg_{-1}$ as follows.
  Recall that $E_{i,j}$ is the $(l+1)\times (l+1)$ matrix whose $(i,j)$-th entry is $1$ and other entries are all $0.$  Write, for convenience,
 $$\mathbf{e}_{i}=E_{l+1,i+1},\hspace{0.15in} \mathbf{f}_{i}=E_{i+1,1} \hspace{0.15in} \mathrm{for}\hspace{0.08in} i=l,...,l-1\hspace{0.08in}.$$ Then $\{\mathbf{e}_1,\mathbf{e}_2...,\mathbf{e}_{l-1}\}$ (resp. $\{\mathbf{f}_1,\mathbf{f}_2,...,\mathbf{f}_{l-1}\}$ ) forms a standard basis of $\mfg_{-1}^{(1)}$ (resp.  $\mfg_{-1}^{(2)}$).

 Let $\{\mathbf{f}_1,...,\mathbf{f}_{l-1},\mathbf{e}_1,...,\mathbf{e}_{l-1}\}$ be an ordered basis of $\mfg_{-1}.$  Recall that there are the relations among generators: \be \label{Eq1} [\mathbf{e}_i,\mathbf{e}_j]=0=[\mathbf{f}_i,\mathbf{f}_j]\hspace{0.1in} \mathrm{for \hspace{0.07in}all}\hspace{0.07in} 1\leq i,j\leq l-1,\ee
 \be \label{Eq2} [\mathbf{f}_i,\mathbf{e}_j]=0 \hspace{0.1in} \mathrm{for \hspace{0.07in}all}\hspace{0.07in} 1\leq i\ne j\leq l-1,\ee

  \be \label{Eq3} [\mathbf{f}_1,\mathbf{e}_1]=[\mathbf{f}_2,\mathbf{e}_2]=\cdots =[\mathbf{f}_{l-1},\mathbf{e}_{l-1}](=\mathbf{h}),\ee where $\mathbf{h}:=E_{l+1,1}$ is the generator of the one dimensional subspace $\mfg_{-2}.$

  Let $\phi\in \mathrm{End}(\mfg_{-1}^{(1)})\oplus \mathrm{End}(\mfg_{-1}^{(2)}).$ In order for $\phi$ to become a derivation $\varphi$ on $\mfm$, i.e., $\iota(\varphi)=\varphi|_{\mfg_{-1}}=\phi$, $\phi$ must satisfy the derivation conditions for the relations (\ref{Eq2}) and (\ref{Eq3}) only. Note that the relations in (\ref{Eq1}) do not impose any restriction on $\phi$ since $\phi$ is an endomorphism preserving the decomposition $\mfg_{-1}=\mfg_{-1}^{(1)}\oplus \mfg_{-1}^{(2)}.$
  \smallskip

  To be explicit, we write $\phi(\mathbf{e}_i)=\sum_{k=1}^{l-1}a_{k,i}\mathbf{e}_k$ and  $\phi(\mathbf{f}_i)=\sum_{k=1}^{l-1}b_{k,i}\mathbf{f}_k$. Then we may consider $a_{i,j}$ and $b_{i,j}$ as coordinate functions of $\mathrm{End}(\mfg_{-1}^{(1)})\oplus \mathrm{End}(\mfg_{-1}^2).$
   From (\ref{Eq2}), we have $$0=\phi([\mathbf{f}_i, \mathbf{e}_j])=[\phi(\mathbf{f}_i), \mathbf{e}_j]+[\mathbf{f}_i,\phi(\mathbf{e}_j)]=(a_{i,j}+b_{j,i})\mathbf{h}.$$
   The last equality follows from the equations (\ref{Eq2}) and (\ref{Eq3}).
   Therefore we get \\
   $(l-1)(l-2)$ equations  $$a_{i,j}+b_{j,i}=0 \hspace{0.1in} ( 1\leq i\ne j\leq l-1).$$

From (\ref{Eq3}), we fix $(l-2)$ relations
 \be \label{eq4} [\mathbf{f}_1,\mathbf{e}_1]=[\mathbf{f}_i,\mathbf{e}_i] \hspace{0.1in}\textrm{for}\hspace{0.08in} i=2,3,...., l-1.\ee
 Applying $\phi$ on both sides of (\ref{eq4}), we get \be \label{eq5} (a_{1,1}+b_{1,1})\mathbf{h}=(a_{i,i}+b_{i,i})\mathbf{h},\ee  which gives $(l-2)$ equations
  $$a_{i,i}+b_{i,i}=a_{1,1}+b_{1,1}  \hspace{0.1in}( i=2,3,...,l-1).
$$ Therefore $\mfg_0^\prime$ is defined by $l(l-2)$ linear equations. We can easily see that all these equations are independent and  hence the codimension of $\mfg_0^\prime$  is $l(l-2)$. Since the dimension of  $\mathrm{End} (\mfg_{-1}^{(1)})\oplus \mathrm{End} (\mfg_{-1}^{(2)})$ is $2(l-1)^{2},$ the dimension of $\mfg_0^\prime$ is $(l-1)^2+1,$ which is exactly the dimension of $\mfg_0.$ Therefore we conclude that
 $\mfg_0^\prime=\mfg_0$. This completes the proof.
\end{proof}

\begin{proposition}\label{prop:type III}
Let $G/P$ be a homogeneous space of type $III$, and let $U$ be a connected open subset of $G/P$.  Then $\mathcal{A}_{\oplus}(U)=\mfg.$
\end{proposition}
\begin{proof}
 The proof is similar to the proof of Proposition \ref{pro:typeII}. Suppose $U\subset G/P$ is an open subset of $M(\mfm)$ (with $o\in U$). Then by the definitions of $\mfg_0^\prime$ and $\mca_\oplus(U)$, for any $x\in U,$ $\psi_{x}:\mfg(\mfm)\hookrightarrow \mathcal{A}(U)$ descends to  $\psi_{x}^\oplus :\mfg(\mfm,\mfg_0^\prime)\hookrightarrow \mathcal{A}_\oplus(U)$. Since $\mfg(\mfm,\mfg_0^\prime)=\mfg(\mfm,\mfg_0)=\mfg$ is finite dimensional,
 as before, by fixing $x$, say, $x=o,$ we obtain  an isomorphism $\psi_{o}:\mfg(\mfm)\DistTo \mathcal{A}(U)$. Therefore we have $$\mca_\oplus(U)=\mca_\oplus(M(\mfm))=\mfg(\mfm,\mfg_0^\prime)=\mfg(\mfm,\mfg_0)=\mfg.$$
The proof for a general connected open subset $U\subset G/P$ is similar to that of Proposition \ref{pro:typeI}.
\end{proof}

\subsection{Extension of a local biholomorphic map}

\begin{proposition}\label{coro1}
Let $\Psi:U_1 \rightarrow U_2$ be a biholomorphic map between two connected open subsets of $G/P$. Suppose that $\Psi$ preserves  $D$, $\mathscr{C}$ or the decomposition $D=D^{(1)}\oplus D^{(2)}$ for  $G/P$ of types $I, II$ or $III$, respectively.
  Then $\Psi$ extends to an automorphism $\widetilde{\Psi}:G/P\rightarrow G/P.$
 \end{proposition}

\begin{proof}
We note that since $\Psi:U_1\rightarrow U_2$ preserves $D$, $\msc,$ or the decomposition $D=D^{(1)}\oplus D^{(2)}$ for types $I, II$  or $III,$ respectively, the differential $d\Psi$ sends  $\mathcal{A}(U_1)$,
 $\mca_\msc(U_1)$ or $\mca_\oplus(U_1)$ isomorphically onto $\mathcal{A}(U_2)$,
 $\mca_\msc(U_2)$ or $\mca_\oplus(U_2)$ for types $I, II$  or $III,$ respectively.
Since $\mca(U_i)$, $\mca_\msc(U_i)$ or $\mca_\oplus(U_i)$ are all  identified with $\mfg$ for types $I$, $II$ or $III$, respectively, by Propositions \ref{pro:typeI}, \ref{pro:typeII} and \ref{prop:type III}, $d\Psi$ sends $\mfg$ isomorphically onto $\mfg$ for all types $I,II, III.$  Fix $x\in U_1$. Put $x^\prime=\Psi(x)$, and let
$\mfp$ and $\mfp^\prime$ be subalgebras of vector fields vanishing at $x$ and  $x^\prime$, respectively.  Then $\mfp$ and $\mfp^\prime$ are parabolic subalgebras of $\mfg$, and $d\Psi:\mfg\rightarrow \mfg$ sends $\mfp$ onto $\mfp^\prime$ isomorphically. Now  let $P$ and $P^\prime$ be the parabolic subgroups of $G$ with Lie algebras $\mfp$ and $\mfp^\prime,$ respectively.  Then the automorphism $d\Psi$ of Lie algebras, in turn, gives rise to an automorphism $G\rightarrow G$ sending $P$ to $P^\prime$ and hence  an automorphism $\widetilde{\Psi}:G/P \rightarrow G/P^\prime$, which naturally extends  $\Psi:U_1\rightarrow U_2.$
 \end{proof}

\section{Main result}\label{Section:main result}
In this section, we will complete the proof of the main theorem.
We begin with the following lemma which will be used throughout this section.
\begin{lemma}
Let $\sgla$ be a simple graded Lie algebra and $G/P$ its associated homogeneous space.
Then the natural action of $G$ on $G/P$ preserves $D$ for all simple graded Lie algebras $\sgla$ including $(A_l,\{\alpha_1,\})$ and $(C_l,\{\alpha_1,\})$. Furthermore,
it preserves $\msc$ or the decomposition $D=D^{(1)}\oplus D^{(2)}$  on $G/P$ for types II or III, respectively.
\end{lemma}
\begin{proof}
Let $h\in G$, and let $x$ and $\tilde{x}$ be points of $G/P$ with $\tilde{x}=h\cdot x$, where $h\cdot x$ denotes the standard action.

Case of $\msc$ for type $II$:\\
 If $f:\P^1\rightarrow G/P$ is a rational curve through $x$ with the image $C,$ then the curve
 $\widetilde{C}=h\cdot C \subset G/P$ is a rational curve through $\tilde{x}= h\cdot x$ parametrized by $\tilde{f}:\P^1 \rightarrow G/P$ defined by
$\tilde{f}(u)=h\cdot f(u).$
Moreover, the differential  $dh:T_x(G/P)\rightarrow T_{\tilde{x}}(G/P)$ sends the tangent of the rational curve $f$ at $x$ to the tangent of the rational curve $\tilde{f}$ at $\tilde{x}$, i.e., $dh$ send isomorphically  $\msc_x$ to $\msc_{\tilde{x}}$. Therefore the natural action of $G$ on $G/P$ preserves $\msc$ for  type $II.$

Cases of $D$ and the decomposition for type $III$: \\
WLOG, we assume that $x=o$ is the base point of $G/P$. Recall that we have the canonical identifications
\smallskip

 \begin{tabular}{c}
$T_o(G/P)=\mfg_{-\mu}\oplus \cdots \oplus \mfg_{-1}$,\\
 $T_{\tilde{o}}(G/P)=h(\mfg_{-\mu}\oplus \cdots \oplus \mfg_{-1})h^{-1}=h\mfg_{\mu}h^{-1}\oplus \cdots\oplus h\mfg_{-1}h^-\hspace{0.08in} \textrm{for}\hspace{0.08in} \tilde{o}=h\cdot o,$\\
\end{tabular}
\smallskip

  \noindent
   and note that, under these identifications, the differential $dh:T_o(G/P)\rightarrow T_{\tilde{o}}(G/P)$  is defined by $dh(X)=hXh^{-1}$. Therefore $dh$ sends  $\mfg_{-1}=D_o$ onto $h\mfg_{-1}h^{-1}=D_{\tilde{o}}$ isomorphically, and hence the natural action of $G$ on $G/P$ preserves $D$. Furthermore, for type $III,$ we have $D_o^{(i)}=\mfg_{-1}^{(i)}$ and $D_{\tilde{o}}^{(i)}=h(\mfg_{-1}^{(i)})h^{-1}$ for $i=1,2,$
  and, by the definition of $dh$, $dh$ sends  isomorphically $D_o^{(i)}$ onto $D_{\tilde{o}}^{(i)}$. Thus the natural action of $G$ on $G/P$ preserves the decomposition $D=D^{(1)}\oplus D^{(2)}$ for type $III.$
\end{proof}

\subsection{ Differential system on $N$ coming from an action of $N$ on $G/P$ }\label{subsection:notation}
Let $A:N\times G/P\rightarrow G/P$ be an EC-structure on $G/P.$ Let $O=O_{A}$ be the Zariski open orbit of the action $A.$ Fix $x\in O$, and define a biregular map $a:N\rightarrow O$ by $a(h)=A(h,x).$
Let $D$ be the natural differential system on $G/P$  and let $$ D^{-\mu}\supset
D^{-\mu+1}\supset\cdots \supset D^{-1}=D $$ be the weak derived system induced by $D.$ Let  $E:=a^*(D|_{O})$ be a differential system on $N$ and let
\be \label{filtration2} E^{-\mu}\supset E^{-\mu+1}\supset \cdots \supset E^{-1}=E\ee
 be the weak derived system induced by the differential system $E$ on $N.$
 Note that the subbundles  $E^k$ of $TN$ are equal to $a^*(D^k|_{O})$ since the morphism $da$ is an isomorphism of Lie algebras of holomorphic vector fields.

 \begin{lemma}\label{lemma-invariant}
 Let $\mfg=\bigoplus_{k\in \Z}\mfg_k$ be a simple graded Lie algebra not isomorphic with  $(C_l, \{\alpha_1\})$ ($l\geq 2$), $(B_l, \{\alpha_l\})$ ($l\geq 3$) or $(G_2, \{\alpha_1\})$, and let $G,P$ and $N$ be its associated groups.  Let $E$ be a differential system on $N$ defined above. Then $E$ is invariant under the natural action of $N$ on $N$.
 \end{lemma}
 \begin{proof}
 Fix $h_1, h_2\in N$ and choose $h\in N$ with $h_2=hh_1.$ Define $l^h:N\rightarrow N$ by $l^h(g)=hg$ for $g\in N.$ To prove the lemma,
 we need to show that the differential $dl^h:T_{h_1}N\rightarrow T_{h_2}N$ sends
 $E_{h_1}$ to $E_{h_2}.$ Let $y_i=A(h_i,x)\in O$ for $i=1,2.$

 Note that the homomorphism of groups $N \rightarrow \mathrm{Aut}(G/P)$ corresponding to  the action $A: N\times G/P\rightarrow G/P$  factors through a homomorphism $N\rightarrow \mathrm{Aut}^0(G/P)=G$ since $N$ is connected. Let $\xi\in G$ be the image of $h$ under this map $N\rightarrow \mathrm{Aut}^0(G/P)=G$. Then  the given action of $h$ on $O$, which corresponds to the natural action of $h$ on $N$, is the same as the natural action of $\xi$ on $O\subset G/P$, i.e.,  $A(h,z)=\xi \cdot z$ for all $z \in O.$ In particular, putting $z=y_1=A(h_1, x)$, we get $y_2=\xi\cdot y_1.$
 Therefore, we have the commutative diagram of isomorphisms of differentials
 $$\xymatrix{T_{h_1}N \ar[r]^-{da_{h_1}}\ar[d]_{dl^h}&T_{y_1}O\ar[d]^{d\xi}\\
T_{h_2}N\ar[r]^-{da_{h_2}}&T_{y_2}O}.$$
 Since the natural action of $G$ on $G/P$ preserves $D$, the differential $d\xi$ sends $D_{y_1}$ onto $D_{y_2}.$
Since we have $dl^h=(da_{h_2})^{-1}\circ d\xi \circ da_{h_1}$, and two differentials  $da_{h_i}$ send $E_{h_i}$ to $D_{y_i}$ for $i=1,2,$ $dl^h$ sends $E_{h_1}$ onto $E_{h_2}$. This proves the lemma.

 \end{proof}

\subsection{Construction of a local biholomorphic map}
Let $\mfm(o)=\bigoplus_{k=-\mu}^{-1}\mfm_k$ be the symbol algebra at the identity $o$ of $N$ associated with the differential system $E$. Let  $\omega:TN\rightarrow \mfn$  be the Maurer-Cartan form on $N$, and $\underline{\mfm}(o)=\bigoplus_{k=-\mu}^{-1}\underline{\mfm}_k$ the graded Lie algebra associated with the filtration on $\mfn$ $$\omega(E^{-\mu})\supset \cdots \supset \omega(E^{-1}).$$ Then since the derived system (\ref{filtration2}) is $N$-invariant by Lemma \ref{lemma-invariant}, the symbol algebra $\mfm(o)$ is nothing but the graded Lie algebra $\underline{\mfm}(o)$. On the other hand,  $\underline{\mfm}(o)$ is isomorphic to $\mfn$ as Lie algebras. Note that there is no canonical isomorphism between $\underline{\mfm}(o)$ and $\mfn.$  Indeed, a choice of a subspace $\mfn_k\subset \omega(E^{k}) $ with $\omega(E^{k})=\mfn_k \oplus \omega(E^{k+1})$ for each $k=-1,...,-\mu$ determines an isomorphism  of (graded) Lie algebras
$$\mfn=\bigoplus_{k=-\mu}^{-1}\mfn_k \stackrel{\sim}{\rightarrow} \underline{\mfm}(o)=\bigoplus_{k=-\mu}^{-1}\underline{\mfm}_k $$ and vice versa.
In the sequel,  by fixing one of such isomorphisms, we will identify three Lie algebras \be \label{equaliy1} \mfn=\underline{\mfm}(o)=\mfm(o).\ee

For $i=1,2,$ let $A_i:N\times G/P\rightarrow G/P$ be an  EC-structure on $G/P.$
For these actions $A_i$, we will use the same notation as for the action $A$ in Subsection \ref{subsection:notation}, only with the subscript $i$. For instance,  for each $i=1,2$, $O_i$ denotes the Zariski open subset of the action $A_i$,  $a_i:N\rightarrow O_i$  a biregular map defined by $a_i(h)=A_i(h,x_i)$, $\mfm_i(o)$  the symbol algebra at the identity $o\in N$ associated with $E_i=a_i^*(D|_{O_i})$ and etc.
To avoid any confusion arising from the notations, we make the notations as simple as possible.
\begin{notation}
Write $A_1(h,x)=h*x,$ $A_2(h,x)=h\star x$ and $A_0(h,x)=h\cdot x$ for $h\in N$ and $x\in G/P$, where $A_0$ is the natural action of $N$ on $G/P.$ For $h\in N$ (or $G$), define $$\phi^{i,h}: G/P\rightarrow G/P$$ by $\phi^{i,h}(x)=A_i(h,x)$ for $i=0,1,2.$  For $h\in N,$  recall that $l^h:N\rightarrow N$ is the map defined by the left multiplication by $h$.
\end{notation}

\begin{proposition} \label{prop3} Let $\mfg=\bigoplus_{k\in \Z}\mfg_k$ be a simple graded Lie algebra not isomorphic with $(C_l, \{\alpha_1\})$ ($l\geq 2$), $(B_l, \{\alpha_l\})$ ($l\geq 3$) or $(G_2, \{\alpha_1\})$,  and $G/P$  its associated homogeneous space.
 For $i=1,2$, let $A_i: N\times G/P\rightarrow G/P$ be an EC-structure on $G/P$ of $N$.
Then there is a group automorphism $F:N\rightarrow N$ such that

 \bn

 \item
 $dF_h\circ dl^h=dl^{F(h)}\circ dF_o$ for all $h\in N$,
 \item the biholomorphic map $\Psi:O_1\rightarrow O_2$ defined by $\Psi=a_2\circ F\circ a_1^{-1}$ satisfies
\bn
\item $\Psi(x_1)=x_2,$
\item $\Psi(h*x_1)= F(h)\star x_2$ for all $h\in N,$
\item for each $u\in O_1,$ the differential $d\Psi_{u}$ sends $D_u$ to $D_{\Psi(u)}$. Furthermore, it sends $\mathscr{C}_u$ to $\mathscr{C}_{\Psi(u)}$ for $G/P$ of type $II$, and  respects the decomposition $D=D^{(1)}\oplus D^{(2)}$ for $G/P$ of type $III$.
\en
\en
\end{proposition}

\begin{proof} The proof is an adaption of Proposition $2.4$ of \cite{FH2}.
For $i=1,2,$ fix an identification $\mfn=\mfm_i(o)$ of Lie algebras as in (\ref{equaliy1}).
Take $g\in G$ such that $x_2=g\cdot x_1$. Then the differential $d\phi^{0,g}=d\phi^{0,g}_{x_1}:\mfm(x_1)\rightarrow \mfm(x_2)$ is an isomorphism of graded Lie algebras that  sends $D_{x_1}$ to $D_{x_2}$ and $\mathscr{C}_{x_1}$ to $\mathscr{C}_{x_2}$. In particular, it respects the decomposition for $G/P$ of type $III.$ Define a Lie algebra isomorphism $f:\mfn \rightarrow \mfn$ by  $f=da_2^{-1}\circ d\phi^{0,g}\circ d a_1$, where $da_i$ denotes $da_i=d(a_i)_o:\mfn=\mfm_i(o)\rightarrow \mfm(x_i)$ for $i=1,2.$ Then since $N$ is a simply connected Lie group, there is an automorphism $F$ such that $dF_o=f.$  Define $\Psi:O_1\rightarrow O_2$ by
 $\Psi=a_2\circ F\circ a_1^{-1}$. Now we will show that $\Psi$ satisfies the properties $(1)$ and $(2).$
 \smallskip

 $(1)$ is immediate from the automorphism property $F\circ l^h(g)=l^{F(h)}\circ F(g)$ for all $g,h\in N.$
$(a)$ of $(2)$ is trivial to check.  $(b)$ of $(2)$ follows from the equalities
$$\Psi(h*x_1)=a_2\circ F\circ a_1^{-1}(h*x_1)=a_2\circ F(h \hspace{0.03in} o)$$
$$=a_2(F(h) \hspace{0.03in}o)=F(h)*a_2(o)=F(h)*x_2.$$
Here $o$ is the identity of $N.$ The second and fourth equalities follow from the equivariance of $a_1$ and $a_2$, respectively, and the others from the definitions.

For $(c),$ fix $u\in O_1$ and let $h\in N$ such that $u=h*x_1$, i.e., $h=a_1^{-1}(u).$  Then we have
$$d\Psi_u(D_u)=da_2\circ dF_h\circ da_1^{-1}(D_{h* x_1})
=da_2\circ dF_h\circ da_1^{-1} (d\phi^{1,h}(D_{x_1}))\stackrel{(*)}{=}$$
$$da_2\circ dF_h\circ dl^h\circ da_1^{-1}(D_{x_1})\stackrel{(**)}{=}da_2\circ dl^{F(h)}\circ dF_o\circ da_1^{-1}(D_{x_1})=$$
$$da_2\circ dl^{F(h)}\circ f\circ da_1^{-1}(D_{x_1})\stackrel{(***)}{=}da_2\circ dl^{F(h)}\circ da_2^{-1}\circ d\phi^{0,g}(D_{x_1})\stackrel{(\star)}{=}$$
$$da_2\circ dl^{F(h)}\circ da_2^{-1}(D_{x_2})\stackrel{(\star\star)}{=}D_{F(h)\star x_2} \stackrel{(\star\star\star)}{=}D_{\Psi(u)}.$$
Here the equality $(*)$ follows from the equality $a_1\circ l^h=\phi^{1,h}\circ a_1$ and the fact that $D$ is invariant under the action $A_1$. The equalities  $(**)$ and $(***)$ follow from $(1)$ and the definition of $f$, respectively.
The equality $(\star)$ follows since $D$ is invariant under the natural action of $G$ on $G/P$. Note that for $y\in O_2$,  $a_2\circ l^{F(h)}\circ a_2^{-1}(y)=F(h)\star y$. Therefore since $D$ is invariant under the action $(h,y)\rightarrow F(h) \star y$, the equality $(\star\star )$ follows. The equality $(\star\star\star)$ is immediate from  $(b)$ of $(2)$. This proves the first statement of  $(c)$ of $(2)$.

\smallskip
To prove the second statement of $(c)$ of $(2)$, we note that in proving the first statement of $(c)$ of $(2)$ which involves $D$, we used only the fact that $D$ is invariant under any action of $N$ on $G/P$. Since $\mathscr{C}$ (resp. the decomposition $D=D^{(1)}\oplus D^{(2)}$) also is invariant under any action of $N$ for $G/P$ of type $II$ (resp. type $III$) by Lemma \ref{lemma-invariant},   the proof for the differential system $D$ works verbatim for VMRT (resp. the decomposition $D=D^{(1)}\oplus D^{(2)}$) for $G/P$ of type $II$ (resp. type $III$). This completes the proof of the proposition.
\end{proof}

\subsection{For cases where $\mathrm{Aut}^0(G/P)\ne G$}\label{exception}
When we proved  Lemma \ref{lemma-invariant} and Proposition \ref{prop3},  essential in proving our main theorem, we needed the condition  $\mathrm{Aut}^0(G/P)= G$.
Now we explain what happens to the simple graded Lie algebras $\sgla$ with $\mathrm{Aut}^0(G/P)\ne G$.
Recall from Subsection \ref{subsection-auto} that the simple graded Lie algebra $\sgla$ for which $\mathrm{Aut}^0(G/P)\ne G$ are $$ (C_l, \{\alpha_1\}) \hspace{0.04in}(l\geq 2), \hspace{0.08in}(B_l, \{\alpha_l\})\hspace{0.04in} (l\geq 3),\hspace{0.08in} \textrm{or} \hspace{0.08in}(G_2, \{\alpha_1\}),$$ and there are corresponding  the simple graded Lie algebras $\tilde{\mfg}$ such that  $\widetilde{G}=\mathrm{Aut}^0(G/P)$ and  $\widetilde{G}/\widetilde{P}$ is isomorphic to $G/P$; $$(A_{2l-1}\{\alpha_1\}),\hspace{0.08in} (D_{l+1},\{\alpha_{l+1}\}) \hspace{0.08in}\textrm{and} \hspace{0.08in}(B_3,\{\alpha_1\}).$$

 We remark that even though two simple graded Lie algebras $\mfg$ and $\tilde{\mfg}$ give the isomorphic homogeneous spaces, they produce non-isomorphic differential systems $D$ and $\widetilde{D}$ on $G/P$ and $\widetilde{G}/\widetilde{P}$, respectively. For instance, for $\mfg=(C_l,\{\alpha_1\})$ (this case  is not considered in our main theorem!), $D_o$ is not isomorphic with $\widetilde{D}_o$ as the diagrams below shows.

 \begin{center}
\tiny{\begin{tikzpicture}[scale=0.9]

   \draw (0,0) rectangle (3,3);
   \draw [very thin](0,2.5)-- (3,2.5);
   \draw [very thin](0,0.5)-- (3,0.5);
   \draw [very thin](0.5,0)-- (0.5,3);
   \draw [very thin](2.5,0)-- (2.5,3);
   \draw (1.5,1.5) node{$0$};
   \draw (0.25,2.75) node{$0$};
   \draw (2.75,0.25) node{$0$};
   \draw (2.75,1.5) node{$1$};
   \draw (1.5,2.75) node{$1$};
   \draw (0.25,3.2) node{$1$};
    \draw (1.5,3.2) node{$2l-2$};
    \draw (2.75,3.2) node{$1$};
    \draw (3.8,1.5) node{for $C_l\hspace{0.02in};$};
    \filldraw[fill=gray!60](0,0) rectangle (0.5,0.5);
    \filldraw[fill=gray!60](0,0.5) rectangle (0.5,2.5);
    \filldraw[fill=gray!20](0.5,0) rectangle (2.5,0.5);
    \draw (0.25,1.5) node{$-1$};
     \draw (2.75,2.75) node{$2$};
     \draw (0.25,0.25) node{$-2$};
      \draw (1.5,0.25) node{$-1$};

    \draw (6,0) rectangle (9,3);
   \draw [very thin](6,2.5)-- (9,2.5);
   \draw [very thin](6.5,0)-- (6.5,3);
  \draw (6.25,3.2) node{$1$};
    \draw (7.75,3.2) node{$2l-1$};
    \draw (10,1.5) node{for $A_{2l-1}.$};
 \filldraw[fill=gray!60](6,0) rectangle (6.5,2.5);
  \draw (7.75,1.25) node{$0$};
   \draw (6.25,2.75) node{$0$};
   \draw (6.25,1.25) node{$-1$};
   \draw (7.75,2.75) node{$1$};
    \end{tikzpicture}}
\end{center}

 This difference in the differential systems  makes things worse for the exceptions; $\widetilde{G}=\mathrm{Aut}^0(G/P)$ does not preserve $D$, and hence,  for a (connected) subgroup $H\subset G$, an action of $H$ on $G/P$ does not necessarily preserve $D$.
 In the above example, it is not difficult to take an element  $g\in \widetilde{G}_0\subset \widetilde{P}\subset \widetilde{G}=PSL_{2l}$ such that $g$ does not preserve $D_o=\mfg_{-1}$ (but $\widetilde{D}_o=\tilde{\mfg}_{-1})$, where the parabolic subgroup $\widetilde{P}\subset \widetilde{G}$ is the group of automorphisms fixing the base point $o\in G/P,$ and $\widetilde{G}_0$ is the group of automorphisms preserving the gradation $\tilde{\mfg}=\tilde{\mfg}_{-1}\oplus\tilde{\mfg}_0\oplus \tilde{\mfg}_1.$
 (Due to this `pathology' arising on $G/P$ for the exception cases, we will treat them separately when we prove the main theorem below.) However we observe that, despite this pathology with the exception case, there is a `nice' embedding of algebraic groups \\ $J:(G,P)\hookrightarrow(\widetilde{G},\widetilde{P})$ (or, equivalently, an embedding $J:(\mfg,\mfp)\hookrightarrow (\tilde{\mfg},\tilde{\mfp})$ of Lie algebras) such that
 \begin{enumerate}
 \item $J(P)=\widetilde{P}\cap J(G)$,
 \item there is an isomorphism of groups $\rho:J(N)\DistTo \widetilde{N}$ (or an isomorphism of Lie algebras $\rho:J(\mfn)\DistTo \tilde{\mfn}$),
 \item the induced isomorphism $J:G/P\rightarrow \widetilde{G}/\widetilde{P}$ sends $D$ into $\widetilde{D}$ (not onto).
 \end{enumerate}
 We would like to emphasize that in general the nilpotent group $J(N)$ does not coincide with,  but isomorphic to $\widetilde{N}$, and hence $N$ is isomorphic to $\widetilde{N}$ via $\rho\circ J$; see the exmaple right below.  We also note that  $J:\mfg\rightarrow \tilde{\mfg}$ does not preserve the gradation.

For the  above example, $J$ is  the embedding defined by $J(X)=X$ for $X\in \mfg$, and $\rho\circ J$ can be described informally as follows:  The light shaded $(-1)$ part is completely determined by the dark shaded $(-1)$ part, and hence the dark shaded $(-1)$ and $(-2)$ parts, which corresponds to $\mfm=\mfn^-\subset \mathfrak{sp}(2l)$, is ``isomorphic" via $\rho\circ J$ to the dark shaded $(-1)$ part that corresponds to $\tilde{\mfm}=(\tilde{\mfn})^-\subset \mathfrak{sl}(2l)$.

 Let us take another example; $(B_l,\{\alpha_1\})$. Note that,  in the notations of \cite{Ya}, $\mathfrak{so}(2l+1)$ consists of $(2l+1)\times (2l+1)$ matrices of the form
 \begin{displaymath}X=\left(\begin{array}{ccc}
A & a&B\\
\xi & 0&-a^\prime\\
C&-\xi^\prime&-A^{\prime}\\
\end{array}\right),
 \end{displaymath}
  Here $A, B, C$ are $l\times l$ matrices, and $B$ and $C$ satisfies $B=-B^\prime$, $C=-C^\prime$. $a$ and $\xi$ are column $l$-vector $a=(a_1,...,a_l)^t$ and row $l$-vector $\xi=(\xi_1,...,\xi_l)$ respectively such that $a^\prime=(a_l,...,a_1)$ and $\xi^\prime=(\xi_l,\xi_{l-1},...,\xi_1)^t$. Note that
   $\mathfrak{so}(2l)$ consists of $2l\times 2l$-matrices of the form  $X^0$, where $X^0$ is a matrix obtained from $X\in \mathfrak{so}(2l+1)$ by removing the center column and row.
  Given $X\in \mathfrak{so}(2l+1)$ above, let $\widetilde{X}$ be a $(2l+2)\times(2l+2)$-matrix defined by
 \begin{displaymath}
\widetilde{X}=\left(\begin{array}{cccc}
A & a& a&B\\
\xi & 0&0&-a^\prime\\
\xi & 0&0&-a^\prime\\
C&-\xi^\prime&-\xi^\prime &-A^{\prime}\\
\end{array}\right).
\end{displaymath}
Then we see that $\widetilde{X}$ belongs to $\mathfrak{so}(2l+2).$
Associating $\widetilde{X}$ to $X$ gives the desired embedding $J:\mfg=\mathfrak{so}(2l+1)\hookrightarrow \tilde{\mfg}=\mathfrak{so}(2l+2)$.
In this case, $J(\mfn)=\tilde{\mfn}$, i.e., $\rho:J(\mfm)\rightarrow \widetilde{\mfm}$ is the identity map.

We will use the above isomorphisms $J:G/P\DistTo \widetilde{G}/\widetilde{P}$ and $\rho\circ J:N \rightarrow \widetilde{N}$ to treat EC-structures on $G/P$ for exception cases. This idea was suggested by the referee.
\subsection{Completion of the proof of the main theorem}

\begin{theorem}[Main Theorem]\label{Theorem-Main} Let $\mfg=\bigoplus_{k\in \Z}\mfg_k$ be a simple graded Lie algebra not isomorphic with $(A_l,\{\alpha_1\})$, $(C_l,\{\alpha_1\})$, and $G/P$ its associated homogeneous space.
 Then there exists, up to isomorphism, a unique EC-structure of $N$ on $G/P$.
\end{theorem}
\begin{proof}
Since  there is an EC-structure of $N$ on $G/P$ coming from the Bruhat decomposition of $G/P$ (\cite{BL1}), it remains to prove uniqueness. We consider two cases separately.

\smallskip
Case $1$: $\sgla$ not isomorphic with $(B_l, \{\alpha_l\})$  or $(G_2, \{\alpha_1\})$:\\
Let $A_i:N\times G/P\rightarrow G/P$ be two EC-structures on $G/P$ for $i=1,2.$ Then by Proposition \ref{prop3}, there is a biholomorphism $\Psi:O_1\rightarrow O_2$ such that the differential $d \Psi$ preserves  $D$, $\mathscr{C}$ or the decomposition $D=D^{(1)} \oplus D^{(2)}$ for $G/P$ of types $I,II$ or $III$, respectively. Then by Proposition \ref{coro1}, $\Psi$ extends to an automorphism $\widetilde{\Psi}:G/P\rightarrow G/P.$ Now we have to check that the automorphism $\widetilde{\Psi}$ of $G/P$ is an isomorphism  between two actions $A_i:N\times G/P\rightarrow G/P$, extending the isomorphism $\Psi$ between $ A_i:N\times O_i\rightarrow O_i$. Since
we have $$\widetilde{\Psi}\circ A_1|_{N\times O_1}=A_2\circ(F\times \widetilde{\Psi})|_{N\times O_1},$$ and $N\times O_1$ is Zariski open in $N\times G/P,$ the equality $\widetilde{\Psi}\circ A_1=A_2\circ(F\times \widetilde{\Psi})$ holds on $N\times G/P$, which means that two EC-structures $A_i$ on $G/P$ are isomorphic.

\smallskip
Case $2$: $\sgla$ isomorphic with$(B_l, \{\alpha_l\})$  or $(G_2, \{\alpha_1\})$:\\
Recall that Lemma \ref{lemma-invariant} and Proposition \ref{prop3} do not work for this case, but from Subsection \ref{exception} we have the  isomorphisms $J:G/P\DistTo \widetilde{G}/\widetilde{P}$ and $\rho\circ J:N \rightarrow \widetilde{N}$. These two isomorphisms induce a bijective correspondence between   actions of $N$ on $G/P$  and actions of $\widetilde{N}$ on $\widetilde{G}/\widetilde{P}.$ Therefore, uniqueness of EC-structures of $N$ on $G/P$ follows from that of $\widetilde{N}$ on $\widetilde{G}/\widetilde{P}.$
\end{proof}

We remark that our result implies that if  $N_1$ and $N_2$ are subgroup of $G$ isomorphic to $N$, and
act on $G/P$  with an open  orbit, then $N_1$ and $N_2$  are conjugate. To see this, for $i=1,2,$ let $A_i: N\times G/P \rightarrow G/P$ be an EC-structure on $G/P$ given via a homomorphism $\varphi_i:N\rightarrow N_i\subset G\subset \mathrm{Aut}^0(G/P)$. Then the automorphism  $\widetilde{\Psi}:G/P\rightarrow G/P$ is given by an action of an element $h\in G.$ In the notations of the proof of Proposition \ref{coro1}, this is because $\widetilde{\Psi}$ is induced by a group isomorphism $(G,P)\rightarrow (G,P^\prime)$ and this isomorphism, in turn, is given as the conjugation by the element $h$ such that $P^\prime =h P h^{-1}.$
Thus, letting $F^\prime:N_1\rightarrow N_2$ be a map  given via the bihomorphic map $\Psi:O_1\rightarrow O_2$, $F^\prime$ is an isomorphism of group given as the conjugation by the element $h$. Note that  for the construction $\Psi:O_1\rightarrow O_2$, it is enough to hypothesize in Proposition \ref{prop3} that the actions $A_i$ are given via homomorphisms $\varphi_i:N\rightarrow G\subset \mathrm{Aut}^0(G/P)$. This is a slightly weaker condition than  $\mathrm{Aut}^0(G/P)=G $. Since in this case $N_i$ are already included in $G$,  this result holds for all simple graded Lie algebras $\sgla$.

\subsection{For projective spaces}
Combining the result of Hassett and Tschinkel(\cite{HT}) and our result, the projective space $\P^n$ is only a homogeneous space that allows more than two EC-structures of a relevant nilpotent group $N$:
If $n\geq 2,$ then there are many EC-structures of $\mathbb{G}_a^n$ on $\P^n$.
Let us give concrete examples of  non-isomorphic EC-structures for $n=2$ given in \cite{HT}.  In this case,  there are, up to isomorphism,  exactly two EC-structures on $\P^2$ (associated with the simple graded Lie algebra $(A_2, \{\alpha_1\})$).
They are given through two homomorphisms $M_1, M_2:\mathbb{G}_{a}^2\stackrel{\sim}{=}N\rightarrow SL(3)$ defined by
\begin{displaymath}M_1(a_1,a_2):= \left(\begin{array}{ccc}
1 & a_1&a_2\\
0 & 1&0\\
0&0&1
\end{array}\right),
\end{displaymath}

\begin{displaymath}M_2(a_1,a_2):= \left(\begin{array}{ccc}
1 & a_1&a_2+\frac{1}{2}a_1^2\\
0 & 1&a_1\\
0&0&1
\end{array}\right).
\end{displaymath}
Here, our unipotent radical $N$ of $P$ consists of matrices $M_1(a_1,a_2)$
for $(a_1,a_2)\in \C^2$.
\smallskip

If we view the projective space $\P^{2l-1}$  as associated with $\mfg=(C_l,\{\alpha_1\})$, then, as in the proof of Theorem \ref{Theorem-Main},  two isomorphisms $J:G/P\DistTo \widetilde{G}/\widetilde{P}$ and $\rho\circ J:N\DistTo \widetilde{N}$  give rise to one-to-one correspondence between EC-structures of $N$ on $G/P$ and EC-structures of $\widetilde{N}$  on $\widetilde{G}/\widetilde{P},$ where $\tilde{\mfg}=(A_{2l-1},\{\alpha_1\}).$


\begin{thebibliography}{50}


 \bibitem{Ar1} I. V. Arzhantsev,  {\it Flag varieties as equivariant compactifications of $\mathbb{G}_a^n$}, Proc. Amer. Math. Soc.  139  (2011),  no. 3, 783-786.
 \bibitem{AS1} I. V. Arzhantsev and E. Sharoyko, {\it Hassett-Tschinkel correspondence: Modality and projective hypersurfaces}, Journal of Alg. 348 (2011), 217-232.
 \bibitem{CN} A. \v{C}ap and K. Neusser, {\it On automorphism groups of some types of generic distributions}, Differential Geom. Appl.  27  (2009),  no. 6, 769-779.
 \bibitem{De1} R. Devyatov, {\it Unipotent commutative group actions on flag varieties and nilpotent multiplications}, Transformation Groups 20 (2015), no. 1, 21-64.
 \bibitem{Do1} B. Doubrov, B. Komrakov and T. Morimoto, {\it Equivalence of holonomic differential equations},  Lobachevskii Journal of Math. Vol 3 (1999), 39-71.
\bibitem{FH1} B. Fu and J.-M. Hwang, {\it Classification of non-degenerate projecrive varieties with nonzero prolongation and application to target rigidity}, Invent. Math. 189 (2012), 457-523.
\bibitem{FH2} B. Fu and J.-M. Hwang, {\it Uniqueness of equivariant compactification of $\C^n$ by a Fano manifold of Picard number 1}, Math. Res. Lett.  21  (2014),  no. 1, 121-125.
\bibitem{Hw1} J.-M. Hwang, {\it Rigidity of homogeneous contact manifolds under Fano deformation}, J. reine angew. Math. 486 (1997), 153-163.
\bibitem{Hw2} J.-M. Hwang, {\it Geometry of Minimal Rational Curves on Fano Manifolds}, ICTP Lecture Notes 6 (2001), 335-393.
\bibitem{HT} B. Hassett and Y. Tschinkel, {\it Geometry of equivariant compactifications of $\mathbb{G}_a^n$}, Internat. Math. Res. Notices 22 (1999), 1211-1230.

\bibitem{Ko1} J. Koll\'{a}r, {\it Rational curves on algebraic varieties}, Ergebnisse Math. 3 Folge 32 Springer (1996).
\bibitem{Mo1}S. Mori, {\it Projective manifolds with ample tangent bundles}, Annals of Math. 110 (1979), 593-606.
\bibitem{On} A. L. Onishchik, {\it On compact Lie groups transtive on certain manifolds}, Sov. Math. Dokl. 1 (1961), 1288-1291. MR0150238 (27:3740).
\bibitem{BL1} S. Billy and V. Lakshimibai, {\it Singular Loci of Schubert Varieties}, Progress in Math. (Boston, Mass.), vo. 182.
\bibitem{T2}N. Tanaka, {\it On differential systems, graded Lie algebras and Pseudo-groups}, J. Math. Kyoto Univ., 10 (1970), 1-82.
\bibitem{Ya} K. Yamaguchi, {\it Differential systems associated with simple graded Lie algebras}, Adv. Stud. Pure Math. 22 (1993), 413-494.

\end{thebibliography}
\end{document}